\documentclass[A4, 11pt]{article}
\usepackage[top=1in,bottom=1in,left=1in,right=1in]{geometry}

\usepackage[sumlimits,intlimits]{amsmath}
\usepackage{amsthm,amssymb,amsfonts}
\usepackage{graphicx,color}

\usepackage{float}

\usepackage[
linktocpage=true,
colorlinks=true,
linkcolor=blue,
citecolor=blue,
urlcolor=blue
]{hyperref}
\usepackage[numbers,sort&compress]{natbib}

\allowdisplaybreaks

\newtheorem{definition}{Definition}[section]

\newtheorem{lemma}[definition]{Lemma}
\newtheorem{theorem}[definition]{Theorem}

\newtheorem*{remark*}{Remark}

\theoremstyle{definition}

\numberwithin{equation}{section}
\numberwithin{figure}{section}

\begin{document}

%


    %
    %

\title{Convergence of Lebenberg-Marquard method for the Inverse Problem with an Interior Measurement}

\author{Yu Jiang$^{1}$, Gen Nakamura$^{2}$
\\$^1$School of Mathematics, Shanghai University of Finance and Economics,\\
 Shanghai 200433, P.R. China\,
\\E-mail: jiang.yu@mail.shufe.edu.cn
\\$^2$Department of Mathematics, Hokkaido University,\\
Sapporo 060-0810, Japan\,
\\E-mail: nakamuragenn@gmail.com
}
\date{}
  
\maketitle

\begin{abstract} The convergence of Levenberg-Marquard method is discussed for the inverse problem to reconstruct the storage modulus and loss modulus for the so called scalar model by single interior measurement. The scalar model is the most simplest model for data analysis used as the modeling partial differential equation in the diagnosing modality called the magnetic resonance elastography which is used to diagnose for instance lever cancer.  The convergence of the method is proved by showing that the measurement map which maps the above unknown moduli to the measured data satisfies the so called the tangential cone condition. The argument of the proof is quite general and in principle can be applied to any similar inverse problem to reconstruct the unknown coefficients of the model equation given as a partial differential equation of divergence form by one single interior measurement. The performance of the method is numerically tested for the two layered picewise homogneneous scalar model in a rectangular domain.
\end{abstract}

{\bf Key words} Levenberg-Marquard method, convergence, tangential cone condition, interior measurement, magnetic resonance elastography, storage modulus, loss modulus

{\bf Key words} 35R30, 65N21

\section{Introduction}\label{Introduction}
Let $\Omega\subset \mathbb{R}^n$ ($n=2\,\mbox{or}\,3)$ be a bounded domain with Lipschitz smooth boundary $\partial \Omega$ and
\[\label{eq::gamma}
\gamma(x):=G'(x)+iG''(x)\,\text{with }i=\sqrt{-1},
\]
where $G'(x)$ and $G''(x)$ are real valued bounded measurable functions on $\Omega$ which satisfy the positivity conditions:
\[\label{eq::positivity}
0<\lambda_1\leq G',\,G''\leq \lambda_2\,\,(\mbox{a.e.}\,x\in\Omega)
\]
with some positive constants $\lambda_1$ and $\lambda_2$. Further, let $\rho(x)\in L^\infty(\Omega)$ satisfy $0<\delta_1\leq \rho(x)\leq \delta_2$ $(\mbox{a.e.}\,x\in \Omega$) with some positive constants $\delta_1$ and $\delta_2$.

It is well known that for any given Dirichlet data $g(x)\in H^{1/2}(\partial \Omega)$, there exists a unique $u(x)\in H^1(\Omega)$ to the boundary value problem
\begin{equation}\label{eq::fp}
\left \{
\begin{array}{ll}
\nabla\cdot\big [\gamma(x)\nabla u(x)\big ]+\rho\omega^2u(x)=0 &\mbox{in } \Omega\\
u(x)=g(x)  &\mbox{on } \partial \Omega,
\end{array}
\right.
\end{equation}
where $\omega>0$ is a given constant. We note that this follows from the positivity of  $-\nabla\cdot(G''\nabla\cdot)$ with Dirichlet boundary condition. It could be the positivity of $-\nabla\cdot(G'\nabla\cdot)$ with Dirichlet condition if there is no lower order term in the equation of  \eqref{eq::fp}. We will refer this as {\it positivity}. For further argument on this
see Chapter 3 of \cite{Mizohata} or \cite{jiangsiam} for even more details. Also concerning the Lipschitz smoothness of $\partial\Omega$ for our boundary value problem, see \cite{Grisvard}.

This boundary value problem \eqref{eq::fp} is the simplest model called the {\it scalar model} for a recent diagnosing modality called MRE (Magnetic Resonance Elastography, see for example \cite{muth,Mad} ) in which $u(x)$ describes a component of the displacement vector of a shear wave with attenuation in a human tissue. The equation of  \eqref{eq::fp} is sometimes called scalar model for MRE. $G'$ and $G''$ are called the {\it storage} and {\it loss} moduli of the tissue. Further $\rho$ describes the density of the tissue which can be taken equal to that of water i.e. $\rho=1000 \,\mbox{\rm kg}/\mbox{m}^3$ and $\omega,\,g$ are the frequency and a component of the displacement vector input to the human body. In the rest of this paper we assume that $\rho$ is equal to the above constant for simplicity.

The hardware of MRE consists of a MRI and vibration system. A time harmonic vibration excited by this vibration system is syncronized to a pulse sequence of MRI so that MRE can measure the displacement vector of a shear wave inside a human tissue. The above $u$ and $g$ in \eqref{eq::fp} are the component of the displacement vector of time harmonic vibration in $\Omega$ and at $\partial\Omega$, respectively. Especially the diplacement $g$ at $\partial\Omega$ is given  by a probe attached at some part of $\partial\Omega$ connected with a bar made of glass fiber-reinforced plastic (GFRP) which propagates the vibration excited by the vibration system of MRI (see \cite{fujisaki}). Away from the place the probe is attached we should have to give Neumann boundary condition. But to simplify the description, we just consider Dirichlet boundary condition given on the whole $\partial\Omega$. This so called elastogram of MRE is to recover $G^\prime$ from the {\it MRE measured data} $u(x)\,(x\in\Omega)$. This is an inverse problem with single interior measurement. A similar inverse problem can be seen in mathematically ideal form of inverse problem for ground water hydrology (\cite{hanke}).

The importance of MRE is that it can realize doctors' palpation inside a human body which had been dreamed by doctors for a long time. Although the hardware of MRE is developing very quickly, the elastogram has not yet developed enough and there are so many challenging questions for elastogram. For further details of MRE and its elastogram, we can refer to, for example, \cite{jiangsiam, Amm} for mathematical modeling, \cite{NJ,Bal, AmmS, Hig} for theorical inversion analysis and \cite{AmmN,Seo,houten04,JiangP,mcn2} for numerical reconstruction schemes.

The precise formulation of this inverse problem is as follows.

\medskip
\noindent {\bf Inverse Problem}:

Recover $\gamma$ (i.e. the storage modulus $G' $ and loss modulus $G''$) from the {\it MRE measured data} $u(x)$ excited by given boundary input $g$ when the frequency $\omega$ is known.

\medskip
Here it should be remarked that our MRE measured data is a single interior measurment. Likewise for any inverse problem, the basic questions for this inverse problem are the uniqueness, stability and reconstruction of identifying $G'$ from the MRE measured data. There isn't any complete uniqueness and stability. What have been known for them so far is as follows for the case $\gamma\in C^1(\overline\Omega)$ with given non-identically zero input $g\in H^{3/2}(\partial\Omega)$ and when $\gamma\big|_{\partial\Omega}$ is known. That is suppose there are finitely disjoint closed analytic manifolds of codimension one which are curves for $n=2$ and surfaces for $n=3$ compactly embedded inside $\Omega$ and $\gamma\in C^1(\overline\Omega)$ is analytic inside and outside these manifolds. Further, $\gamma$ can be extended analytically up to these manifolds from inside and outside of them. If a perturbed $\widetilde\gamma$ of $\gamma$ satisfies the admissibility condition given by
\[\label{admissibility condition}
|\mbox{Im}(\gamma-\widetilde\gamma)(x)|\le\tan(\kappa)|\mbox{Re}(\gamma-\widetilde\gamma)(x)|\,\,(x\in\overline\Omega),
\]
where $\kappa<(\pi-\sigma)/2$ with a small $\sigma>0$, a local H\"older conditional stability estimate has been recently proved in \cite{HMN}.  The H\"older exponent and constant in the estimate only depend on $\gamma$ and $\Vert\widetilde\gamma\Vert_{C^1(\overline\Omega)}$. As a corollary of this result, if $G''$ is known, then the admissibility condition is satisfied and hence the global uniqueness for identifying $G'$ follows.

Despite the lack of complete uniqueness which exactly fit to our case, we are particulary interesting in a mathematically rigorous reconstruction of $\gamma$ even in the case $\gamma\not\in C^1(\overline\Omega)$ from the practical point of view. As a recostruction scheme to identify $\gamma$ from single interior measurement i.e. just measure
$u\big|_\Omega$, we will give a Newton type regularization scheme called Levenberg-Marquardt iterate and prove that the so-called tangential cone condition (\cite{hanke,KNO}) holds, which is the key to show the convergence of this scheme. We would like to emphasize here that the inverse problem with a single interior measurement for any elliptic equation whose boundary value problem has the above mentioned {\it positivity} satisfies the tangential cone condition. 
We will also provide several numerical tests of our scheme for different cases.

The rest of this paper is organized as follows. In Section 2 we review the Levenbert-Marquardt method for the operator equation $F(x)=y$ in $x$ for a given $y$ from \cite{hanke,KNO}. After this in Section 3, we put our inverse problem into such an operator equation and compute the Fr\'echet derivative of $F$
in Subsection 3.1. Then in Subsection 3.2, we show that $F$ satisfies the tangential cone condition. Section 4 is devoted to the numerical study of
our reconstruction scheme for our inverse problem based on the convergence of Levenberg-Marquadt method.  In final section, some discussions and a conclusion are given.

\section{The Levenberg - Marquardt Method}\label{sec::LMM}
Any nonlinear inverse problem can be treated as a nonlinear operator equation
\begin{equation}\label{eq::nonlinear op}
F(x)=y
\end{equation}
with respect to $x$, where $F:D(F)\subset \mathcal{X}\to \mathcal{Y}$ is a differentiable operator between Hilbert space $\mathcal{X}$ and $\mathcal{Y}$. In practice, we can only measure a noisy data $y^{\delta}$ which satisfies
\[\label{eq::noisy level}
\|y^{\delta}-y\|\leq \delta
\]
with a noise level $\delta$. By knowing $y^{\delta}$, we always need to get a good approximation of the true solution $x^{\dag}$ which satisfies $F(x^{\dag})=y$.

In practice, it is necessary to find some fast method to solve it. Newton type iterative methods are good for this purpose. Having $x_k$ after $k$ iterations, this iteration updates $x_k$ to $x_{k+1}=x_k+h$, by solving
\[\label{eq::linearized}
F'(x_k)h=y^{\delta}-F(x_k)
\]
with respect to $h$, where $F'(\gamma_k)$ is the Fr\'{e}chet derivative of $F$ at $x_k$. Here, we have to concern that these linearized problems are ill-posed. If the Tikhonov regularization is applied to this ill-posed linearized problem by adding the regularization term $\alpha_k\|h\|^2$, we have the {\it Levenberg - Marquardt method} (LM method) with the following iteration procedure (cf. \cite{hanke,KNO})
\[\label{eq::LMmethod}
x_{k+1}=x_k+\big (F'(x_k)^*F'(x_k)+\alpha_k I\big )^{-1}F'(x_k)^*\big(y^{\delta}-F(x_k)\big)\quad (k\in\{0\}\cup \mathbb{N}),
\]
where $x_k$ is uniquely determined by the Morozov descrepancey principle i.e.
$x_{k+1}-x_k$ is the minimum norm solution of
\begin{equation}\label{eq::Morozov}
\Vert y^\delta-F(x_k)-F'(x_k)(x_{k+1}-x_k)\Vert=q\Vert y^\delta-F(x_k)\Vert
\end{equation}
with any fixed $0<q<1$.

Assume that a true solution $x^{\dag}$ exists in an open ball $B_R(x_0)\subset D(F)$ with radius $R>0$ centered at $x_0$ for an {\it initial guess $x_0$}.
Further, we assume that $F'$ is uniformly bounded in $B_R(x_0)$ and satisfies the  {\it tangential cone condition}:
\begin{equation}\label{eq::tcc}
\|F(x)-F(\widetilde{x})-F'(x)(x-\widetilde{x})\|\leq c\|x-\widetilde{x}\|\|F(x)-F(\widetilde{x})\|,
\end{equation}
for any $x$ and $\widetilde{x}\in B_R(x_0)\subset D(F)$. Then, we have the following theorem.

\begin{theorem}\label{thm::convergence1}
	The Levenberg - Marquardt method with exact data $y^{\delta}=y$, $\|x_0-x^{\dag}\|<q/c$ converges to a solution of $F(x)=y$ as $k \to \infty$.
\end{theorem}

This theorem means that a true solution $x^{\dag}\in B_R(x_0)$ of the equation \eqref{eq::nonlinear op} with exact data $y$ can be recovered by the Levenberg - Marquardt method. For the noisy data $y^{\delta}$, we have to set up some stopping rule to terminate the iteration appropriately, i.e. as soon as the step index $k=k_*$ satisfies the following discrepancy principle
\begin{equation}\label{eq::stopping rule}
\|y^{\delta}-F(x_{k_*})\|\leq \tau\delta<\|y^{\delta}-F(x_k)\|,\quad 0\leq k<k_*,
\end{equation}
with a constant $\tau>1/q$, stop the iteration. The following theorem gives a convergence of the Levenberg - Marquardt method for noisy data.

\begin{theorem}\label{thm::convergence2}
	Let $k_*=k_*(\delta,y^{\delta})$ be chosen according to the stopping rule \eqref{eq::stopping rule} with $\tau>1/q$. Then starting from the initial guess $x_0$ which satisfies $\|x_0-x^{\dag}\|\leq (q\tau-1)/(c(1+\tau))$, the discrepancy principle \eqref{eq::stopping rule} terminates the Levenberg - Marquardt method with $\alpha_k$ determined from \eqref{eq::Morozov} after finitely many iterations $k_*$ and
	we have \[k_*(\delta,y^{\delta})=O(1+|\ln \delta|).\]
	Moreover, the sequence $x_k$ $(k=0,1,\cdots)$ of the Levenberg - Marquardt method converges to a solution of the equation $F(x) = y$ as $\delta\to 0$.
\end{theorem}

\section{Application to Inverse Problem of MRE}
Let $u^{\delta}:=u^{obs}\in H^1(\Omega)$ be the MRE measured datum which may have a noise with a noise level $\delta$. Also, let $\mathcal{P}$ be the set defined by
$$
\begin{array}{rl}
\mathcal{P}:=&\{\gamma=G'+iG''|(G',\,G'')\in [L^\infty(\Omega)]^2\}\\
&\cap\{\gamma=G'+iG''|0<\widetilde{G'}\leq G'\leq \widehat{G'},\,0<\widetilde{G''}\leq G''\leq \widehat{G''}\}
\end{array}$$
with some positive constants $\widetilde{G'},\,\,\widehat{G'}$. We consider the operator equation
\begin{equation}\label{eq::operator equation}
F(\gamma)=u,
\end{equation}
where $F$ is defined by
\[\label{eq::operator}
F:\, D(F):=\mathcal{P} \subset L^{\infty}(\Omega) \subset \mathcal{X}:=L^2(\Omega) \to \mathcal{Y}:=H^1(\Omega),
\]
and $u$ is the solution to \eqref{eq::fp} with $\gamma=G'+i G''\in \mathcal{P}$.

Under this setting we will use the Levenberg-Marquardt method given in Section \ref{sec::LMM} which considered the following iteration procedure:
\[\label{eq::newton}
\gamma_{k+1}=\gamma_k+\big(F'(\gamma_k) ^*F'(\gamma_k) +\alpha_k I \big)^{-1}F'(\gamma_k)^*
\big(u^{\delta}-F(\gamma_k)\big)\quad (k\in \mathbb{N}\cup \{0\}).
\]

From what we gave in Section \ref{sec::LMM}, we can have the convegence of this iteration scheme, if we show the Fr\'{e}chet derivative $F'(\gamma)$ at $\gamma=\gamma_k\in \mathcal{P}$ exists and is uniformly bounded in $B_R(\gamma_0)$, and the tangential cone condition \eqref{eq::tcc} holds in  $B_{2R}(\gamma_0)\subset \mathcal{P}$ for some $R>0$ and an initial guess $\gamma_0$.

\subsection{Fr\'{e}chet Derivative}
In this subsection we will compute the Fr\'{e}ched derivative $F'(\gamma)$ of $F$ at $\gamma$. Let
$$\gamma^{\delta}:=\gamma+\delta \gamma\,\mbox{with}\,\,\delta \gamma\in L^\infty(\Omega),
\Vert \delta \gamma\Vert_{L^\infty(\Omega)}<<1$$
be the perturbation of $\gamma$ and $F(\gamma^{\delta}):=u^{\delta}$ be
the corresponding output, where $u^{\delta}$ is the solution to
\begin{equation}\label{eq::pfp}
\left \{
\begin{array}{ll}
\nabla\cdot\big [\gamma^{\delta}\nabla u^{\delta}(x)\big ]+\rho\omega^2u^{\delta}(x)=0 &\mbox{in } \Omega\\
u^{\delta}(x)=g(x)  &\mbox{on } \partial \Omega
\end{array}
\right.
\end{equation}
with $\gamma^{\delta}=\gamma+\delta \gamma$. Note that $u^\delta$ here is not the noisy data of $u$ corresponding to $\gamma$, it is the exact solution corresponding to $\gamma^\delta=\gamma+\delta\gamma$.

\medskip
Then the Fr\'echet derivative $F'(\gamma): \mathcal{P}\subset L^\infty(\Omega)\rightarrow H^1(\Omega)$ is given as follows.

\begin{lemma}\label{lem::frechet}
	$u'=F'(\gamma)\delta \gamma\in H^1(\Omega)$ is a solution to
	\begin{equation}\label{eq::pbvp2}
	\left \{
	\begin{array}{ll}
	\nabla\cdot\big [\gamma\nabla u'\big ]+\rho\omega^2 u'=-\nabla\cdot\big [\delta \gamma\nabla u\big ]&\mbox{\rm in } \Omega\\
	u'=0  &\mbox{\rm on } \partial \Omega.
	\end{array}
	\right.
	\end{equation}
\end{lemma}
\begin{proof}
	By comparing \eqref{eq::fp} and \eqref{eq::pfp}, it is easy to find that $\delta u:=u^{\delta}-u\in H^1(\Omega)$ is the solution to following boundary value problem:
	\begin{equation}\label{eq::pbvp}
	\left \{
	\begin{array}{ll}
	\nabla\cdot\big [(\gamma+\delta \gamma)\nabla \delta u\big ]+\rho\omega^2\delta u=-\nabla\cdot\big [\delta \gamma \nabla u\big ]&\mbox{in } \Omega\\
	\delta u=0  &\mbox{on } \partial \Omega
	\end{array}
	\right.
	\end{equation}
	with $\delta\gamma=\gamma^{\delta}-\gamma$.
	By the standard estimate of solutions of the boundary value problem for elliptic equation satisfying the {\it positivity} which is the positivity of $G''$ if it is not zero and that of $G'$ if $G''=0$ and $\omega$ is small relative to $G'$, $\delta u $ satisfies the estimate
	\[
	\|\delta u\|_{H^1(\Omega)}\lesssim \|\delta \gamma\nabla u\|_{L^2(\Omega)}\lesssim \|g\|_{H^{1/2}(\partial \Omega)}\|\delta \gamma\|_{L^\infty(\Omega)}.
	\]
	Hereafter in this paper, the notation ``$\lesssim$'' denotes the inequality ``$ \le$''
	modulo a multiplication by a positive constant which depends only on $\widetilde{G'}$, $\widehat{G'}$, $\widetilde{G''}$, $\widehat{G''}$, $\lambda_1$, $\lambda_2$, $g$, $\rho$, $\omega$ and $\Omega$.
	From \eqref{eq::pfp} and \eqref{eq::pbvp}, $v:=u^{\delta}-u-u'=\delta u-u'\in H^1(\Omega)$ is a solution to
	\[\label{eq::pbvp3}
	\left \{
	\begin{array}{ll}
	\nabla\cdot\big [\gamma\nabla v\big ]+\rho\omega^2 v=-\nabla\cdot\big [\delta \gamma\nabla \delta u\big ]&\mbox{in } \Omega\\
	v=0  &\mbox{on } \partial \Omega.
	\end{array}
	\right.
	\]
	Again by the standard estimate of solutions of  boundary value problem for elliptic equations with {\it positivity}, $v$ satisfies the estimate
	\[\label{eq::esdeltau}
	\|v\|_{H^1(\Omega)}\lesssim \|\delta \gamma\nabla \delta u\|_{L^2(\Omega)}\lesssim \|\delta \gamma\|_{L^\infty(\Omega)}\|\delta u\|_{H^1(\Omega)}\lesssim \|g\|_{H^{1/2}(\partial \Omega)}\|\delta \gamma\|^2_{L^\infty(\Omega)}.
	\]
	Thus we have
	\begin{equation}
	\|u^{\delta}-u-u'\|_{H^1(\Omega)}=O(\|\delta \gamma\|^2_{L^\infty(\Omega)}),
	\end{equation}
	which implies $F'(\gamma)\delta \gamma=u'$.
\end{proof}

Moreover, by the regularity estimate, we have
\[
\|F'(\gamma)\delta \gamma\|_{H^1(\Omega)}\lesssim\|\delta\gamma\|_{L^\infty(\Omega)}\|g\|_{H^{1/2}(\partial \Omega)},
\]
and this implies $F'$ is uniformly bounded near $\gamma$.

\subsection{Tangential Cone Condition}
We can have the tangential cone condition as follows:
\begin{lemma}\label{lem::cone}
	There is an open ball $B_R(\gamma)\subset \mathcal {P}$ about $\gamma$ of radius $R>0$ such that for all $\widetilde{\gamma},\,\widehat{\gamma}\in B_R(\gamma)$,
	\begin{equation}\label{eq::tccm}
	\|F(\widetilde{\gamma})-F(\widehat{\gamma})-F'(\widehat{\gamma})(\widetilde{\gamma}-\widehat{\gamma})\|_{H^1(\Omega)}\leq c \|\widetilde{\gamma}-\widehat{\gamma}\|_{L^\infty(\Omega)}\|F(\widetilde{\gamma})-F(\widehat{\gamma})\|_{H^1(\Omega)}.
	\end{equation}
	Here we note that $c$ is proportional to the reciprocal of the lower bound of the positivity
	mentioned in the proof of Lemma \ref{lem::frechet}.
\end{lemma}
\begin{proof}
	First of all, by \eqref{eq::pbvp2} and \eqref{eq::pbvp}, we have:
	\[\label{eq::pbvp4}
	\left \{
	\begin{array}{ll}
	\nabla\cdot\big [(\gamma+\delta \gamma)\nabla \delta u\big ]+\rho\omega^2\delta u=\nabla\cdot\big [\gamma\nabla u'\big ]+\rho\omega^2u'&\mbox{in } \Omega\\
	\delta u=0  &\mbox{on } \partial \Omega.
	\end{array}
	\right.
	\]
	This implies the estimate
	\[
	\|u^{\delta}-u\|_{H^1(\Omega)}=\|\delta u\|_{H^1(\Omega)}\lesssim \|\gamma\|_{L^\infty(\Omega)}\|u'\|_{H^1(\Omega)}.
	\]
	By inserting this into \eqref{eq::esdeltau}, we have the estimate
	\begin{equation}\label{eq::udelta}
	\|u^{\delta}-u-u'\|_{H^1(\Omega)}\lesssim \|\gamma\|_{L^\infty(\Omega)} \|\delta \gamma\|_{L^\infty(\Omega)}\|u'\|_{H^1(\Omega)}.
	\end{equation}
	
	Similar to \eqref{eq::udelta}, we find that
	\[	\begin{array}{rl}
	&\|F(\widetilde{\gamma})-F(\widehat{\gamma})-F'(\widehat{\gamma})\delta \gamma\|_{H^1(\Omega)}\\
	\leq& C \|\widehat{\gamma}\|_{L^\infty(\Omega)} \|\delta \gamma\|_{L^\infty(\Omega)}\|F'(\widehat{\gamma})\delta \gamma\|_{H^1(\Omega)}\\
	\leq &C( \|\gamma\|_{L^\infty(\Omega)}+r_0)\|\delta \gamma\|_{L^\infty(\Omega)}\|F'(\widehat{\gamma})\delta \gamma\|_{H^1(\Omega)},
	\end{array}\]
	which holds for all $\widehat{\gamma} \in B_{r_0}(\gamma)$ where $r_0>0$ is such that $B_{r_0}(\gamma)\subset \mathcal{P}$ and $\widetilde{\gamma}=\widehat{\gamma}+\delta \gamma \in \mathcal{P}$. Here $C$ is a positive constant independent of $\widetilde{\gamma}$, $\widehat{\gamma}$ and $\gamma$. If $\|\delta \gamma\|_{L^\infty(\Omega)}<[2C( \|\gamma\|_{L^\infty(\Omega)}+r_0)]^{-1}$ then we have
	\[
	\|F'(\widehat{\gamma})\delta \gamma\|_{H^1(\Omega)}\leq 2\|F(\widetilde{\gamma})-F(\widehat{\gamma})\|_{H^1(\Omega)}
	\]
	and it follows that the estimate
	\[	\begin{array}{rl}
	&\|F(\widetilde{\gamma})-F(\widehat{\gamma})-F'(\widehat{\gamma})(\widetilde{\gamma}-\widehat{\gamma})\|_{H^1(\Omega)}\\
	\leq &2C( \|\gamma\|_{L^\infty(\Omega)}+r_0)\|\widetilde{\gamma}-\widehat{\gamma}\|_{L^\infty(\Omega)}\|F(\widetilde{\gamma})-F(\widehat{\gamma})\|_{H^1(\Omega)}.
	\end{array}\]
	
	Therefore, for any positive $R<\min \left \{r_0,\,[2C( \|\gamma\|_{L^\infty(\Omega)}+r_0)]^{-1}\right \}$, the tangential cone condition \eqref{eq::tccm} always holds true for $c=2C( \|\gamma\|_{L^\infty(\Omega)}+r_0)$.
\end{proof}

Before closing this section, we formulate our main result which we can have by combining Theorem \ref{thm::convergence1} , Theorem \ref{thm::convergence2}  , Lemma \ref{lem::frechet}  and Lemma \ref{lem::cone}.

\begin{theorem}\label{thm::main}
	The Levenberg-Marquad method can be applied to give a reconstruction scheme for our inverse problem formulated as solving the operator equation \eqref{eq::operator equation}
	both for the cases that the data is exact and is inexact with an error.
\end{theorem}

\section{Numerical test of Lebenberg-Marquard method for MRE}
In this section we will numerically test the performance of our Levenberg-Marquad method. The set up for this numerical test is based on our MRE experiments done in Hokkaido University \cite{fujisaki}.
\begin{figure}[H]
	\centering
	\includegraphics[width=7cm]{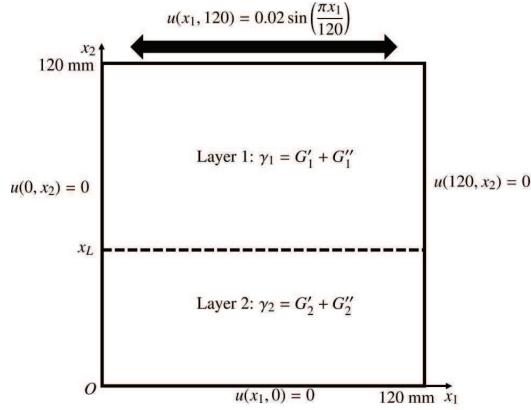}
	\caption{Some setup of parameters for 2D simulations.}\label{fig::fw0}
\end{figure} 
Concerning the spacial resolution of our 0.3 T micro-MRI, which is 1 mm, and the typical size of our two layered agarose gel phantom, which is  $120 \times 120$ grids of a 120 mm $\times$ 120 mm rectangular domain $\Omega$ (see Figure \ref{fig::fw0}),  we consider the following boundary value problem as a special case of \eqref{eq::fp}:
\begin{equation}\label{eq::fpnum}
\left \{
\begin{array}{ll}
(G_1'+iG_1'')\Delta u_1(x_1,x_2)+\rho\omega^2u_1(x_1,x_2)=0\\
\hspace{3cm} 0<x_1<120,\, x_L < x_2<120,\\
(G_2'+iG_2'')\Delta u_2(x_1,x_2)+\rho\omega^2u_2(x_1,x_2)=0\\
\hspace{3cm} 0<x_1<120,\, 0<x_2<x_L,\\
u_1(x_1,x_L)=u_2(x_1,x_L) &0<x_1<120,\\
(G_1'+iG_1'')\partial_{x_2}u_1(x_1,x_L)=(G_2'+iG_2'')\partial_{x_2}u_2(x_1,x_L)&0<x_1<120,\\
u(x_1,120)=u_1(x_1,120)=0.02\sin\left (\dfrac{\pi x_1}{120}\right ) &0<x_1<120,\\
u(x_1,0)=u_2(x_1,0)=0 &0<x_1<120,\\
u(0,x_2)=u_1(0,x_2)=0,\,u(120,x_2)=u_1(120,x_2)=0 &x_L \leq  x_2<120,\\
u(0,x_2)=u_2(0,x_2)=0,\,u(120,x_2)=u_2(120,x_2)=0 &0<x_2<120,
\end{array}
\right.
\end{equation}
with a interface along $x_2=x_L$. Here $(G_1'\,G_1'')$ and $(G_2'\,G_2'')$ are constants.

By applying the method of separation of variables, the solution of \eqref{eq::fpnum} is given as \begin{equation}\label{eq::solution}
u(x_1,x_2)=v(x_2)\times 0.02\sin\left (\dfrac{\pi x_1}{120}\right )
\end{equation}
with
\[
v(x_2)=\left \{
\begin{array}{ll}
c_1e^{i\beta_1 x_2}+c_2 e^{-i\beta_1 x_2} &x_L < x_2<120,\\
d_1e^{i\beta_2 x_2}+d_2 e^{-i\beta_2 x_2}  &0<x_2<x_L,
\end{array}
\right.
\]
where $(c_1,c_2,d_1,d_2)^{\rm{T}}$ is the solution of linear system:
\[
\left [\begin{array}{cccc}
e^{120i\beta_1 } & e^{-120i\beta_1 } & 0 & 0 \\ 
e^{i\beta_1x_L } & e^{-i\beta_1x_L } & -e^{i\beta_2x_L } & -e^{-i\beta_2x_L } \\ 
\beta_1e^{i\beta_1x_L } & -\beta_1e^{-i\beta_1x_L } & -\beta_2e^{i\beta_2x_L } & \beta_2e^{-i\beta_2x_L } \\ 
0 & 0 & 1 & 1
\end{array} \right]
\left [
\begin{array}{c}
c_1 \\ 
c_2 \\ 
d_1 \\ 
d_2
\end{array}\right]=\left [
\begin{array}{c}
1 \\ 
0 \\ 
0\\ 
0
\end{array} \right]
\]
with
\[\beta_j=\sqrt{\dfrac{\rho\omega^2}{G'_j+iG''_j}+\left(\dfrac{\pi}{120}\right)^2}, \quad j=1,2.\]
The other assumptions of parameters used in our numerical simulation are given in Table \ref{table::setup}, 
\begin{table}[H]
	\centering
	\caption{Physical parameters in MRE experiments.}\label{table::setup}
	\centering
	\tabcolsep=8pt
	\begin{tabular}{c||c|c||c|c}
		\hline
		\hline
		Layer 1 ($x_L<x_2<120$) & $G'$ &  20 kPa & $G''$ & 0.4 Pa$\cdot$s$\times \omega$\\
		\hline
		Layer 2  ($0<x_2<x_L$)& $G'$ & 10 kPa & $G''$ & 0.3 Pa$\cdot$s$\times \omega$\\
		\hline
		Others & $\rho$ & 1.0$\times$ 10$^3$ kg/m$^3$ & $x_L$ & 60 mm$\times 2 \pi$\\
		\hline
		\hline
	\end{tabular}
\end{table}

Consider the inverse problem stated in Section \ref{Introduction} for the boundary value problem \eqref{eq::fpnum} under the assumption that we know $x_L$. The numerical test of the performance of our Lebenberg-Marquard method will be given below for this inverse problem. 

\bigskip
\noindent {\bf Example 1: elastic case:}

In this example, we assume $G''= 0$. Recall the uniqueness and stability results of \cite{HMN} given in Section \ref{Introduction}.  It does not exactly fit to the case we have right now. However, due to the fact that we do know the interface and $\Omega$ is just a rectangle, we can easily adapt the argument given in \cite{HMN} to have the uniqueness for our inverse problem in this case.  By using the simulated data generated by \eqref{eq::solution}, we will test our Levenberg-Marquad method in a finite dimensional space obtained by discretizing \eqref{eq::operator equation} for the above setup. We used Matlab$\circledR$ inner-embedded program for the numerical implementation of the method. The progam can adjust automatically the regularizing parameters to ensure the convergence of $\{\gamma_k\}$.

It should be noticed that the iterative sequence may converge to some local minimal point for unsuitable initial iteration guess according to Theorem \ref{thm::convergence1}. Because the constant $c$ in Theorem \ref{thm::convergence1} is highly dependent on $G'$ and $\omega$ when $G''= 0$, we either choose an initial guess of $G'$ which is quite close to the exact value in high frequency case, or can choose an initial guess of $G'$ which is not so close to the exact value in low frequency case. 

\medskip
\noindent {\small\bf Example 1.1: low frequency case:}

In this case, we assume the angular frequency $\omega=20$ Hz and consider the two simulated data. The one is without noise (see Figure \ref{fig::eobs}(a)) and the other is with 20$\%$ relative Gaussian noise (see Figure \ref{fig::eobs}(b) and Figure \ref{fig::eu60} ).

\begin{figure}[H]
	\centering
	(a)\includegraphics[width=.45\textwidth]{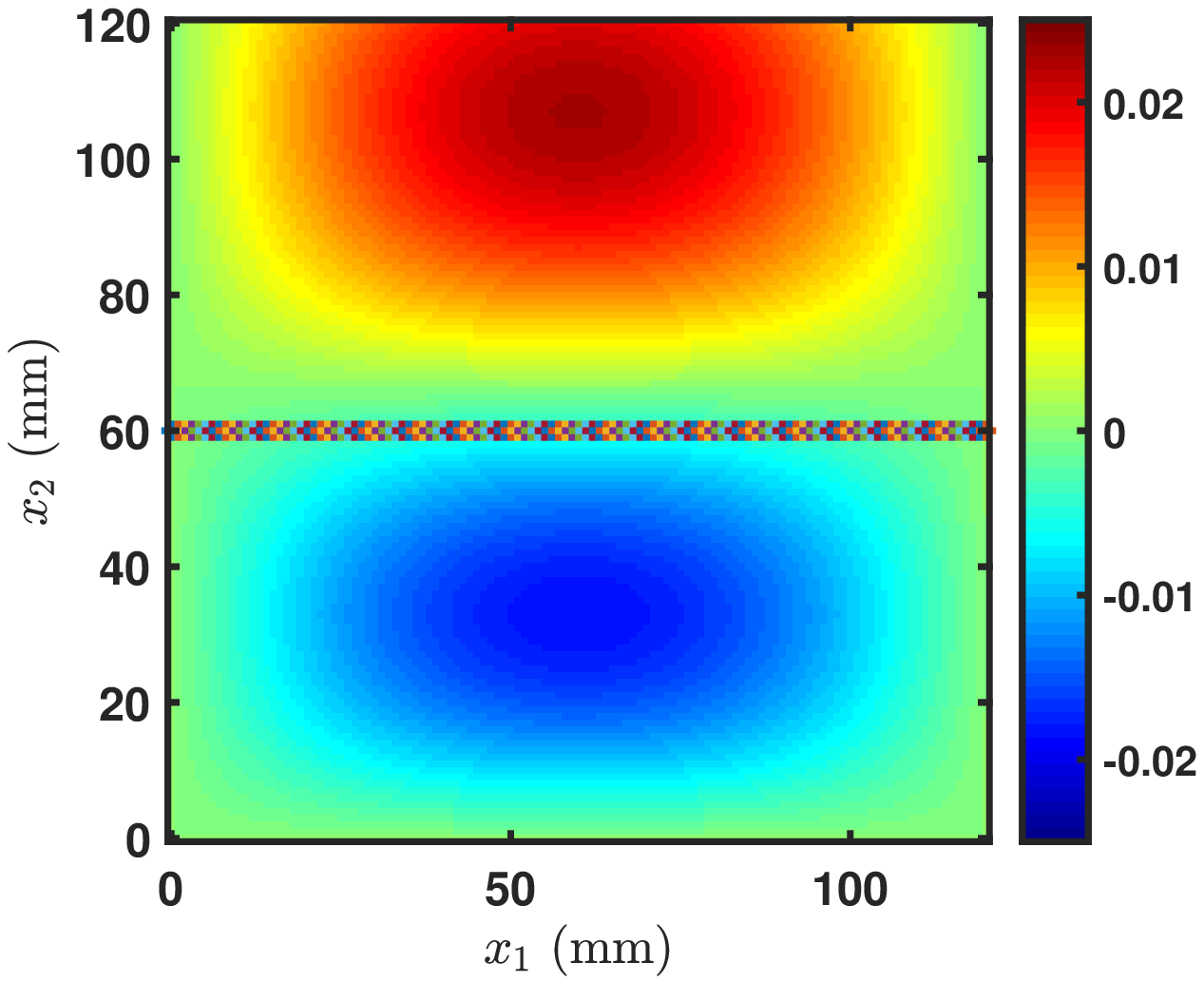}\,
	(b)\includegraphics[width=.45\textwidth]{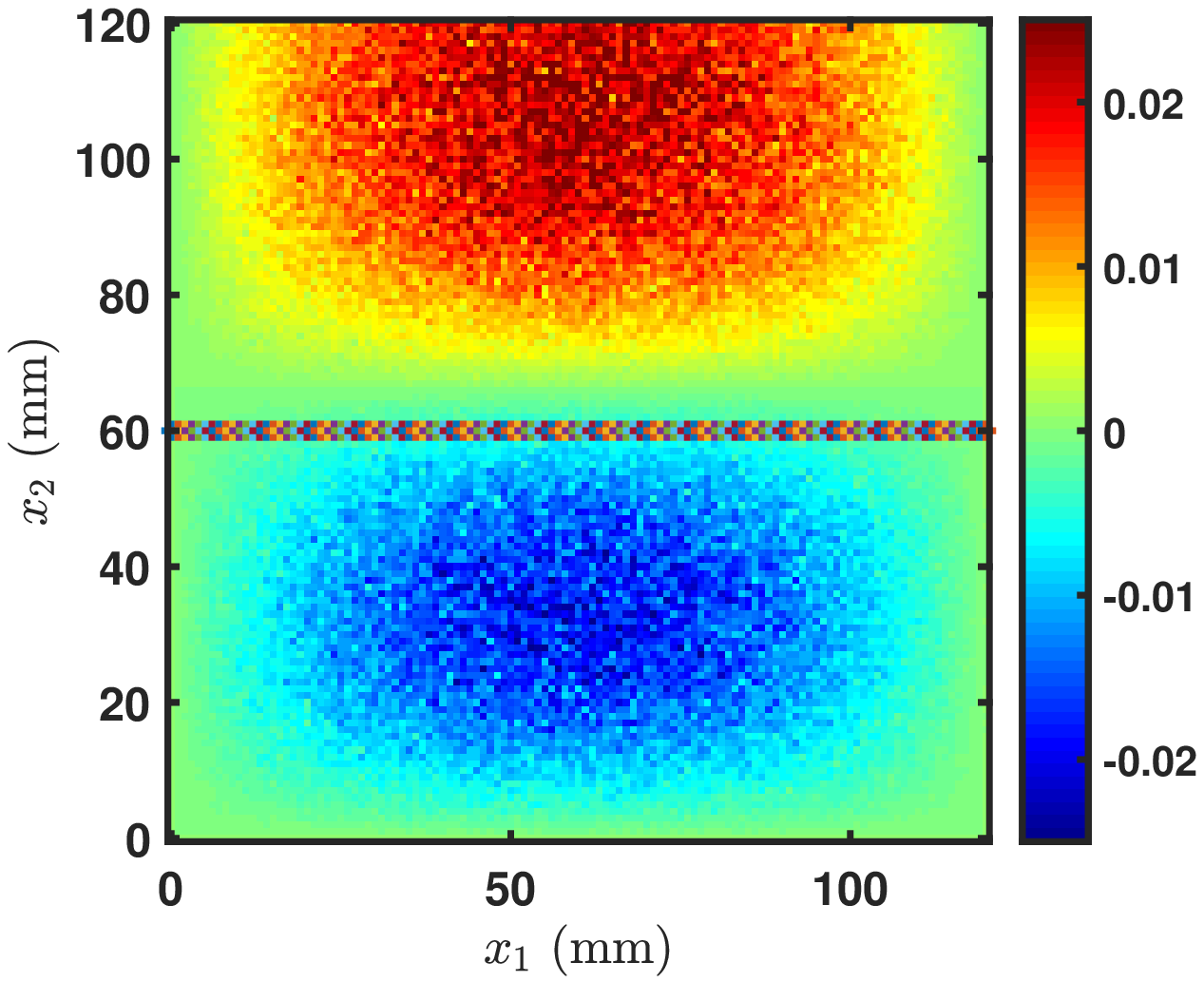}
	\caption{Simulated data: (a) without noise  (mm); (b) with 20$\%$ relative noise (mm).}\label{fig::eobs}
\end{figure}
\begin{figure}[H]
	\centering
	\includegraphics[width=.6\textwidth]{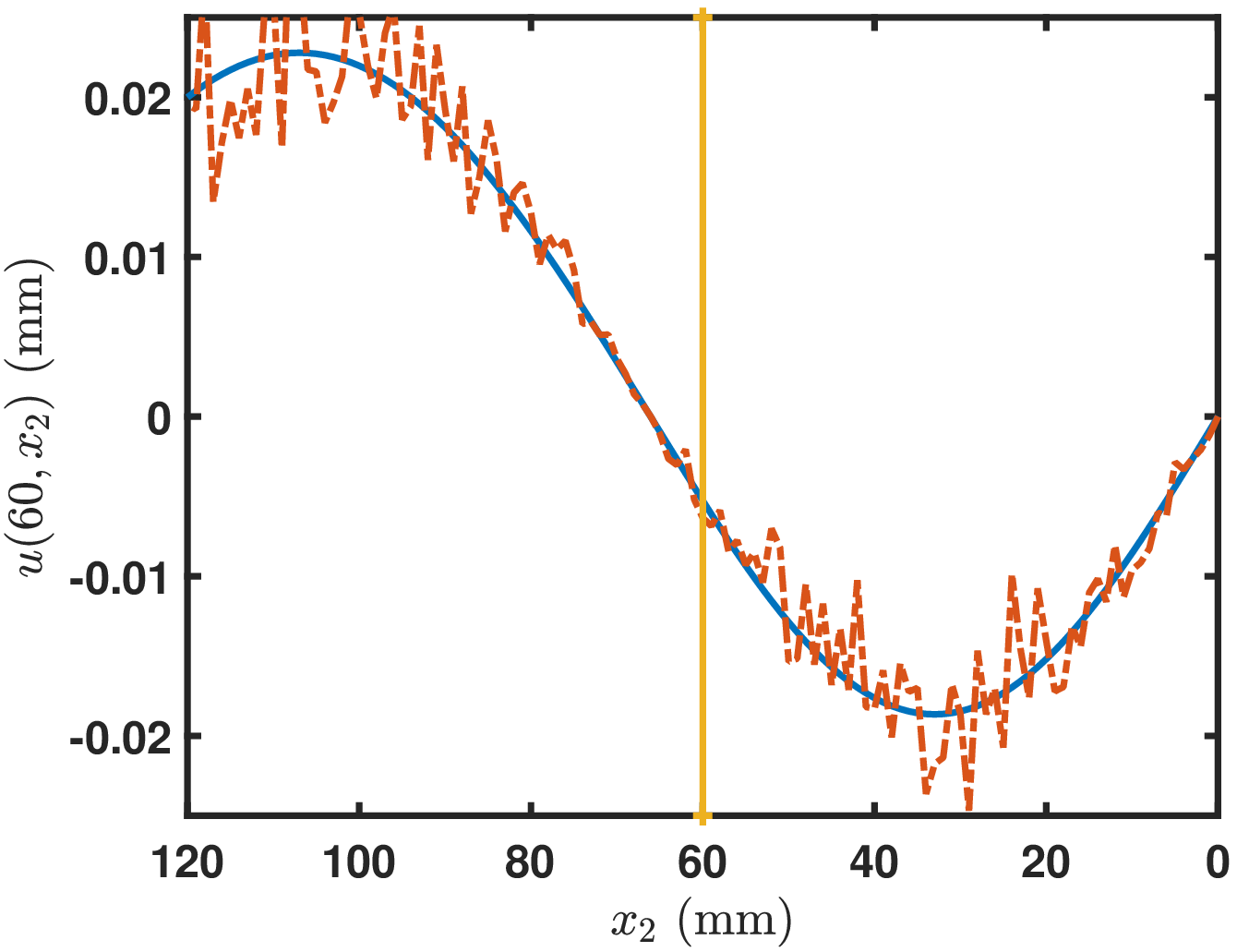}
	\caption{Simulated data along $x_1=60$ mm (blue line: without noise, red line: with 20$\%$ relative noise).}\label{fig::eu60}
\end{figure}

By applying our Levenberg-Marquad method, we recovered $\gamma$ from noisy data shown in Table \ref{table::ereconver} and the reconstructed wave fields are shown in Figure \ref{fig::eobsr}, \ref{fig::eur60}.
\begin{table}[H]
	\caption{Recovery of $G'$.}\label{table::ereconver}
	\centering
	\tabcolsep=10pt
	\begin{tabular}{c||c|c}
		\hline
		\hline	
		Initial guess & $G'$ &  30 kPa \\
		\hline
		\hline
		Layer 1 ($x_L<x_2<120$) & $G'$ &  20.1370 kPa \\
		\hline
		Layer 2  ($0<x_2<x_L$)& $G'$ & 9.9888 kPa \\
		\hline
		\hline
	\end{tabular}
\end{table}
\begin{figure}[H]
	\centering
	\includegraphics[width=.45\textwidth]{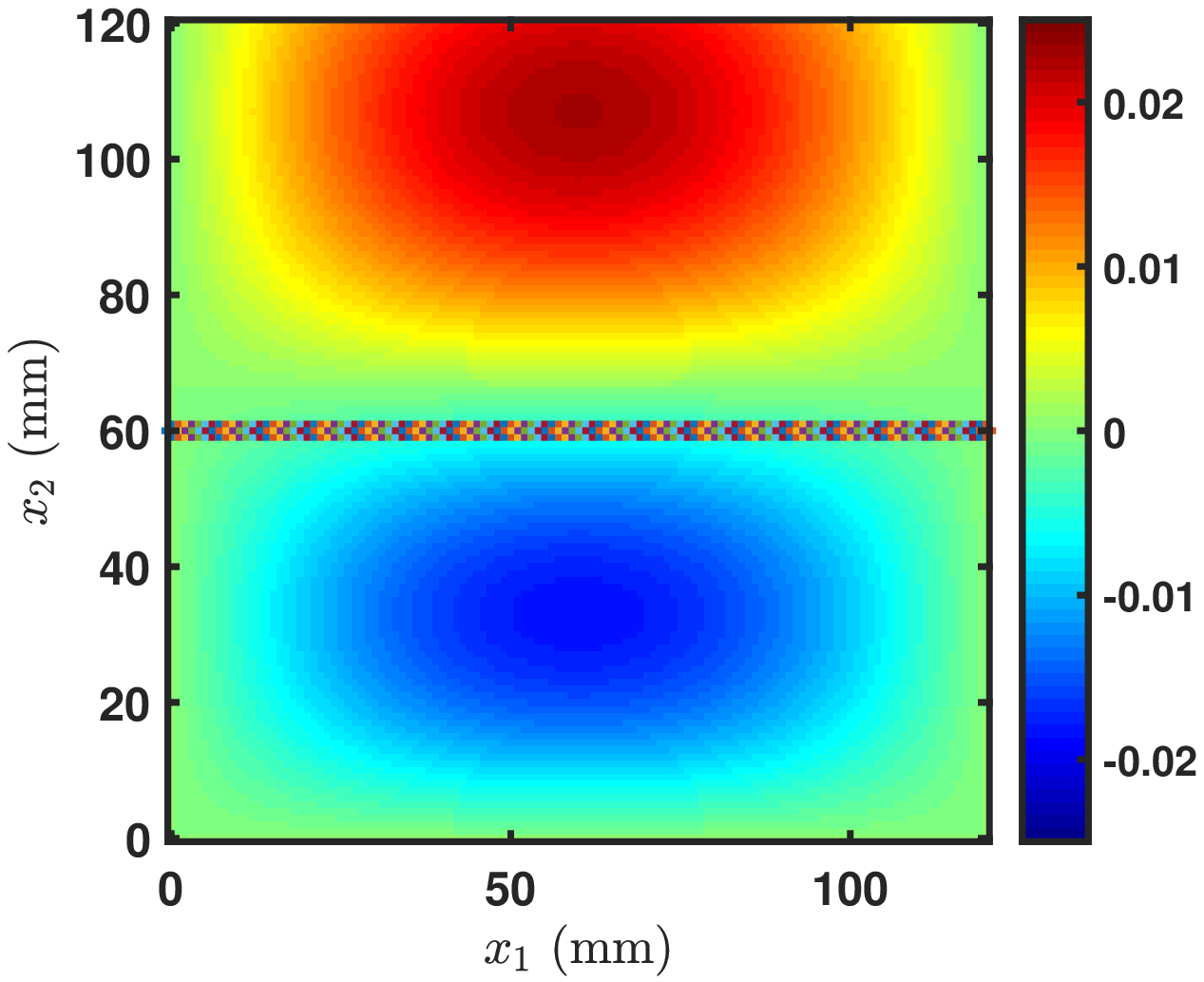}\quad
	\caption{Reconstructed simulated data $u$ (mm).}\label{fig::eobsr}
\end{figure}
\begin{figure}[H]
	\centering
	\includegraphics[width=.6\textwidth]{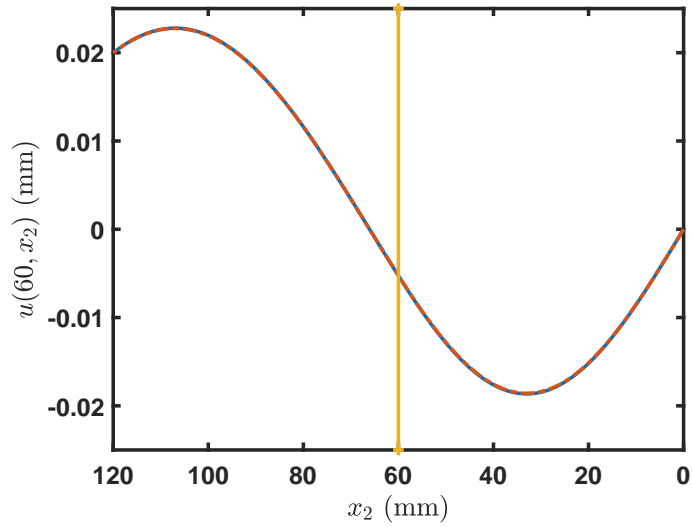}
	\caption{Reconstructed data along $x_1=60$ mm (blue line: without noise, red dot dash line: reconstructed).}\label{fig::eur60}
\end{figure}

\noindent {\small\bf Example 1.2: high frequency case:}
Let the angular frequency $\omega=250$ Hz and consider the two simulated data. The one is without noise (see Figure \ref{fig::heobs}(a)) and the other is with 20$\%$ relative Gaussian noise (see Figure \ref{fig::heobs}(b) and Figure \ref{fig::heu60} ).
\begin{figure}[H]
	\centering
	(a)\includegraphics[width=.45\textwidth]{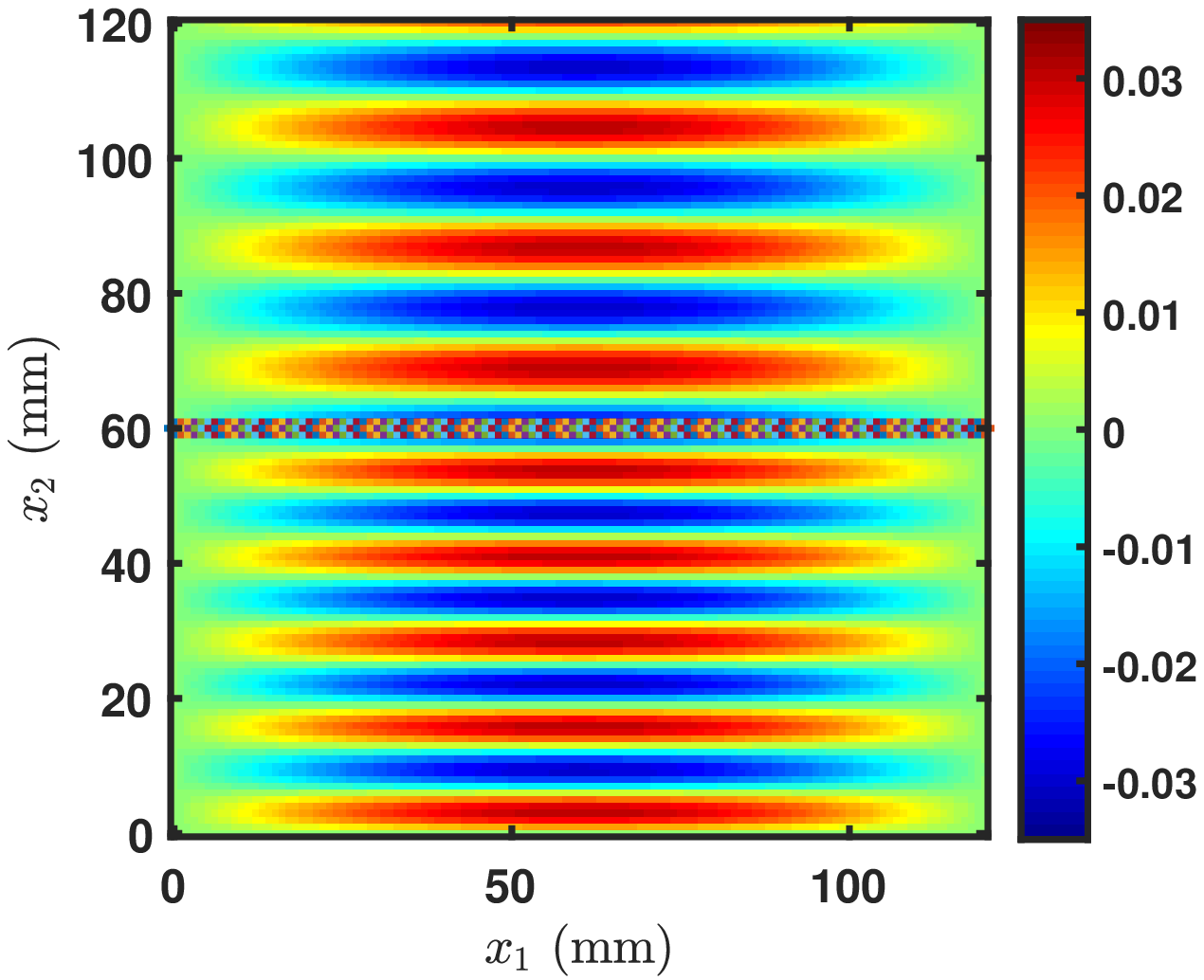}\,
	(b)\includegraphics[width=.45\textwidth]{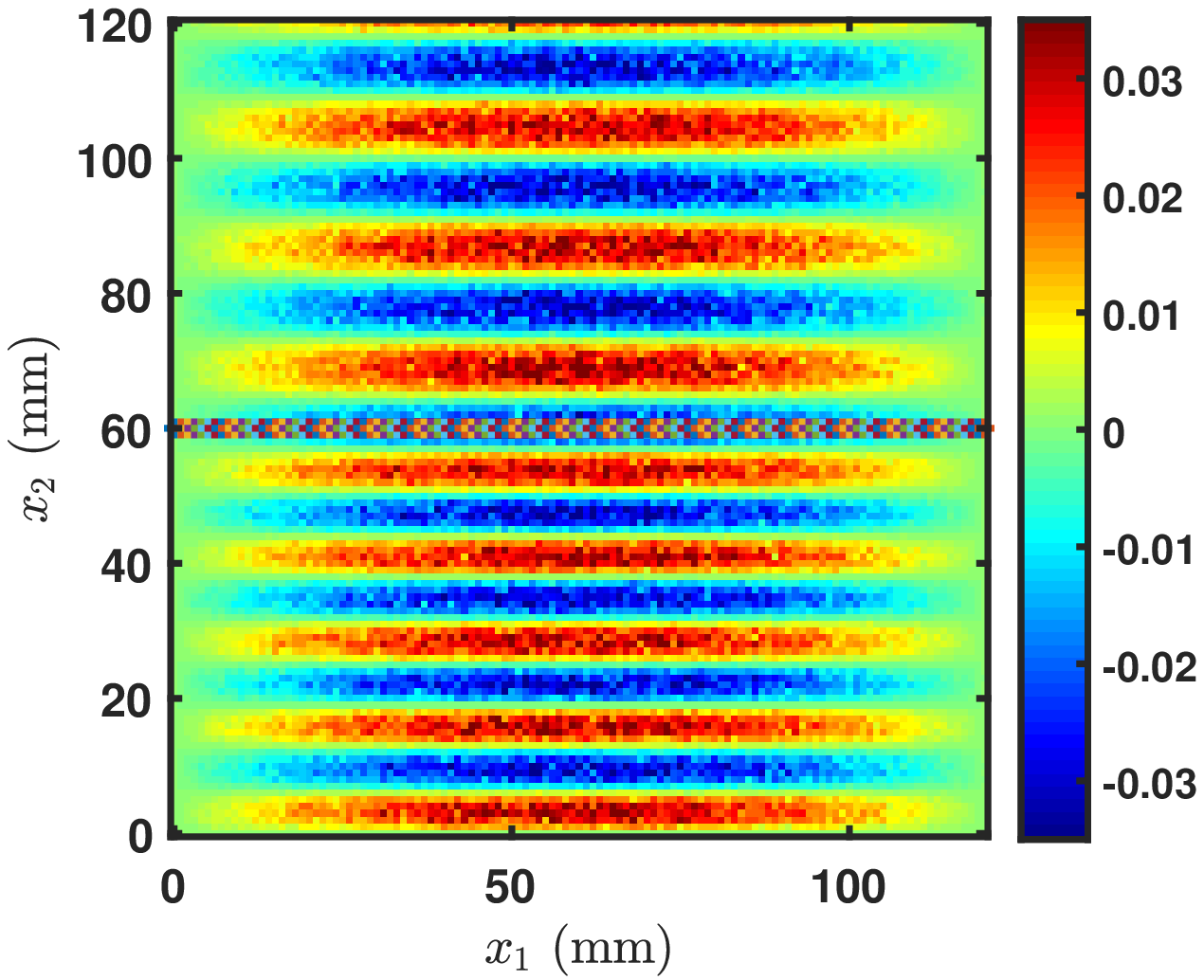}
	\caption{Simulated data: (a) without noise  (mm); (b) with 20$\%$ relative noise (mm).}\label{fig::heobs}
\end{figure}
\begin{figure}[H]
	\centering
	\includegraphics[width=.6\textwidth]{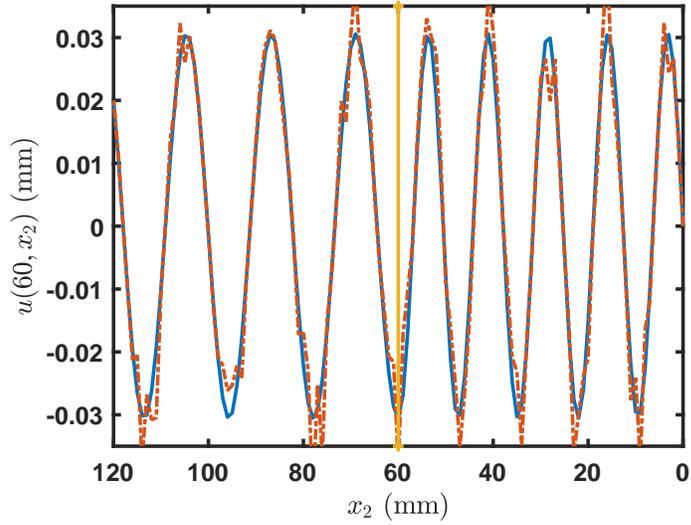}
	\caption{Simulated data along $x_1=60$ mm (blue line: without noise, red line: with 20$\%$ relative noise).}\label{fig::heu60}
\end{figure}

By applying our Levenberg-Marquad method, we recovered  $\gamma$ from the noisy data which is shown in Table \ref{table::hereconver}.  The reconstructed wave fields using the recovered $\gamma$ are shown in Figure \ref{fig::eobsr} and Figure \ref{fig::heur60}. We need to emphasize here that the initial guess of $G'$ must be quite close to the exact value. 
\begin{table}[H]
	\caption{Recovery of $G'$.}\label{table::hereconver}
	\centering
	\tabcolsep=10pt
	\begin{tabular}{c||c|c}
		\hline
		\hline	
		Initial guess & $G'$ &  21 kPa (layer 1) and 9.5 kPa (layer 2) \\
		\hline
		\hline
		Layer 1 ($x_L<x_2<120$) & $G'$ &  19.9987 kPa \\
		\hline
		Layer 2  ($0<x_2<x_L$)& $G'$ & 9.9999 kPa \\
		\hline
		\hline
	\end{tabular}
\end{table}
\begin{figure}[H]
	\centering
	\includegraphics[width=.45\textwidth]{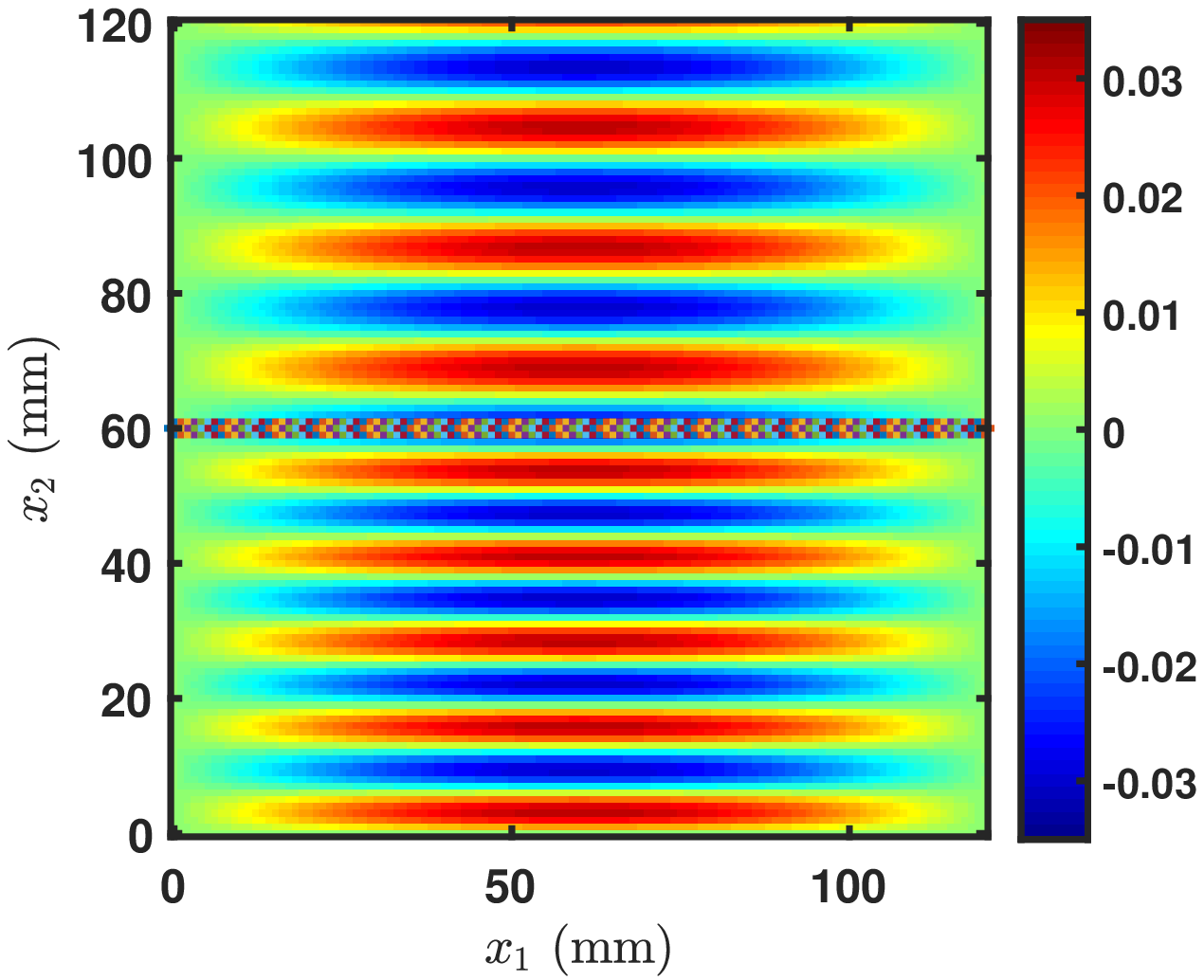}\quad
	\caption{Reconstructed simulated data $u$ (mm).}\label{fig::heobsr}
\end{figure}
\begin{figure}[H]
	\centering
	\includegraphics[width=.6\textwidth]{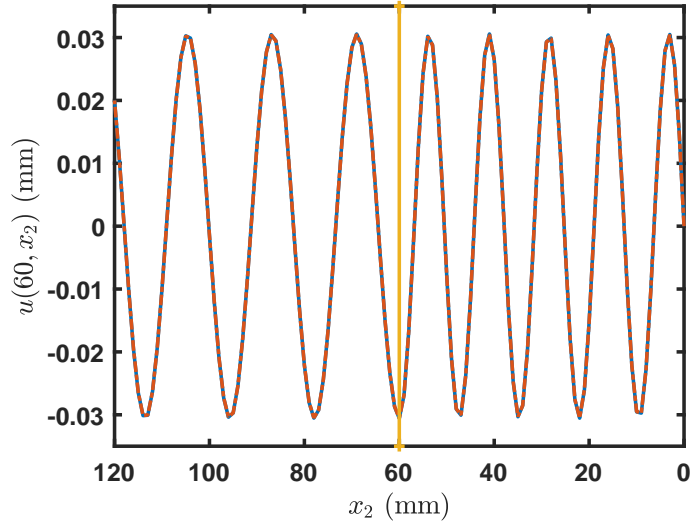}
	\caption{Reconstructed data along $x_1=60$ mm (blue line: without noise, red dot dash line: reconstructed).}\label{fig::heur60}
\end{figure}

\noindent {\bf Example 2: viscoelastic case:}

In this example, we assume $G''\neq 0$ as shown in Table \ref{table::setup}. Due to the positivity of $G''$, we can choose an initial guess of $\gamma$ which is not so close to the exact value for any frequency case according to Theorem \ref{thm::convergence1} and Lemma \ref{lem::cone}. 

\noindent {\small\bf Example 2.1: low frequency case:}

Let the angular frequency $\omega=20$ Hz. Use the simulated data without noise (see Figure \ref{fig::lobs}) and the noisy simulated data with 20$\%$ relative Gaussian noise (see Figure \ref{fig::lobs20} and Figure \ref{fig::lu60} ). 
\begin{figure}[H]
	\centering
	(a)\includegraphics[width=.45\textwidth]{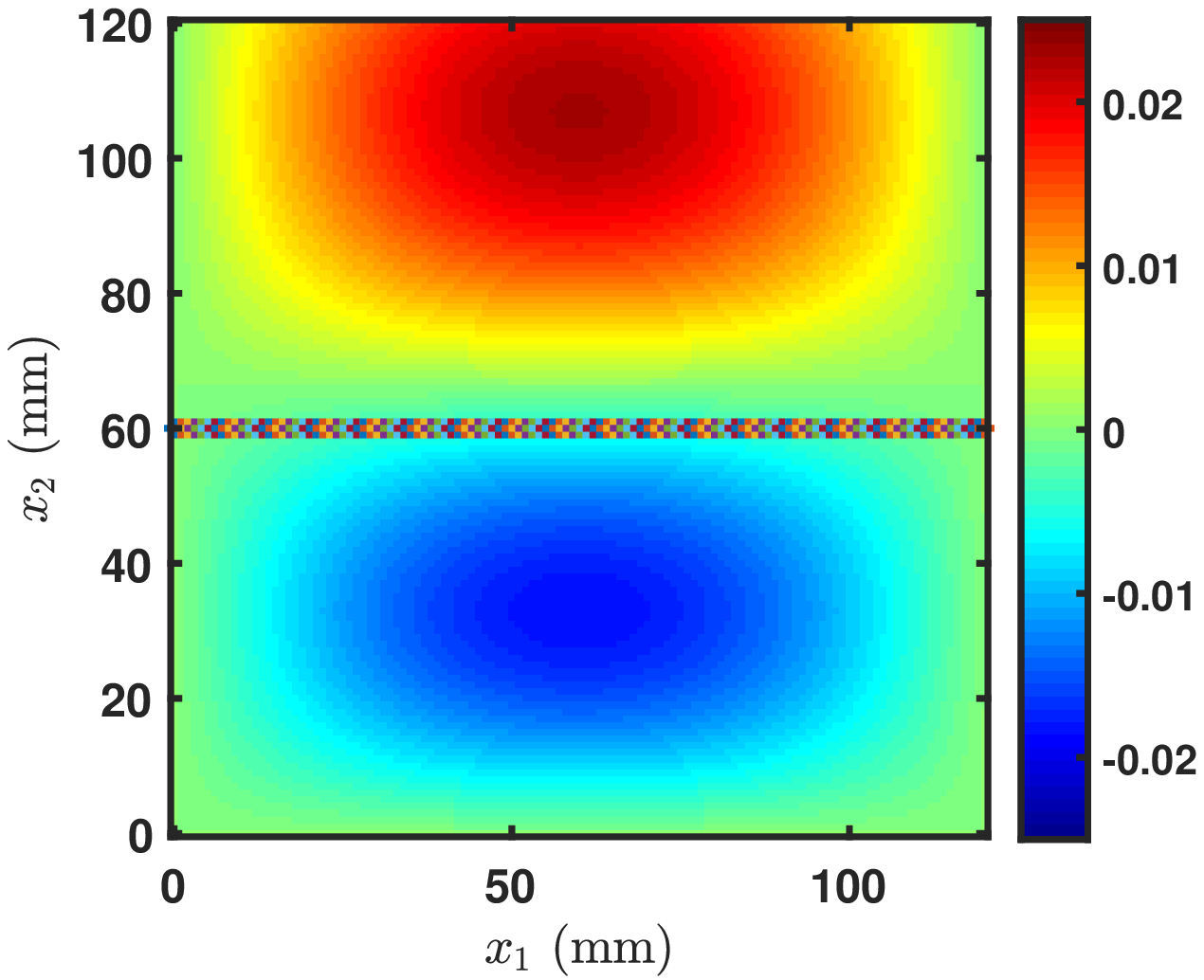}\,
	(b)\includegraphics[width=.45\textwidth]{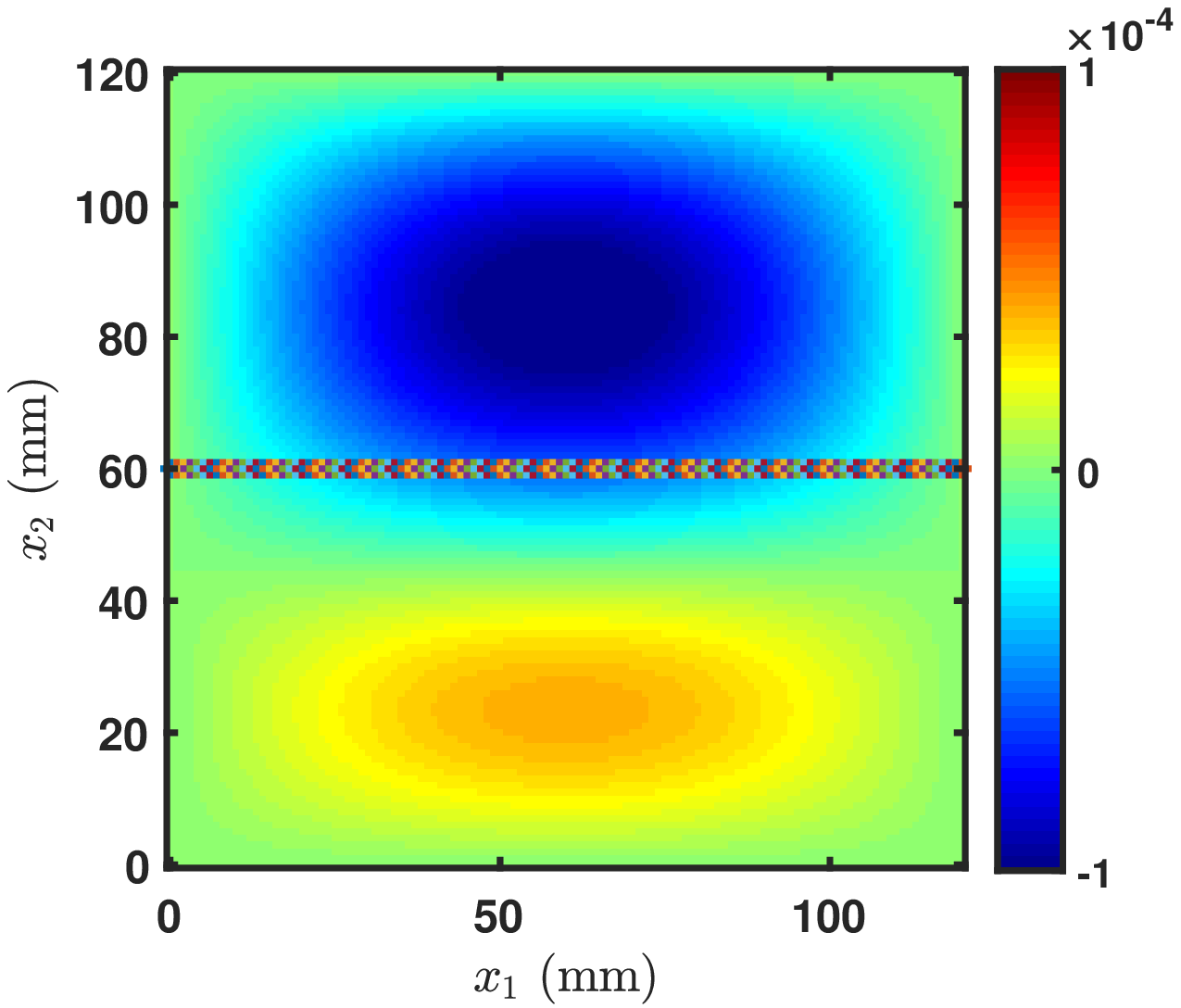}
	\caption{Simulated data without noise: (a) real part of  $u$ (mm); (b) imaginary part of  $u$ (mm).}\label{fig::lobs}
\end{figure}
\begin{figure}[H]
	\centering
	(a)\includegraphics[width=.45\textwidth]{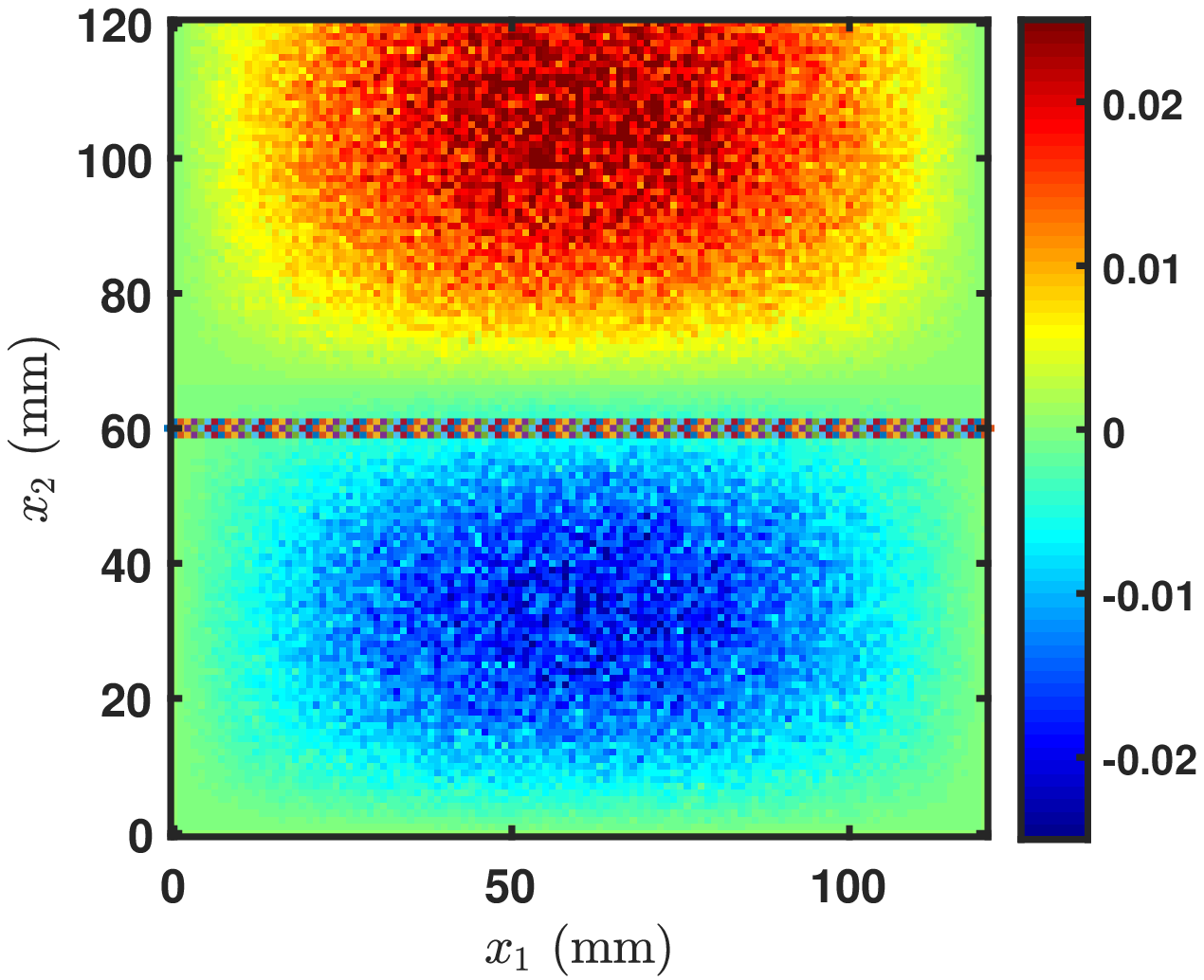}\,
	(b)\includegraphics[width=.45\textwidth]{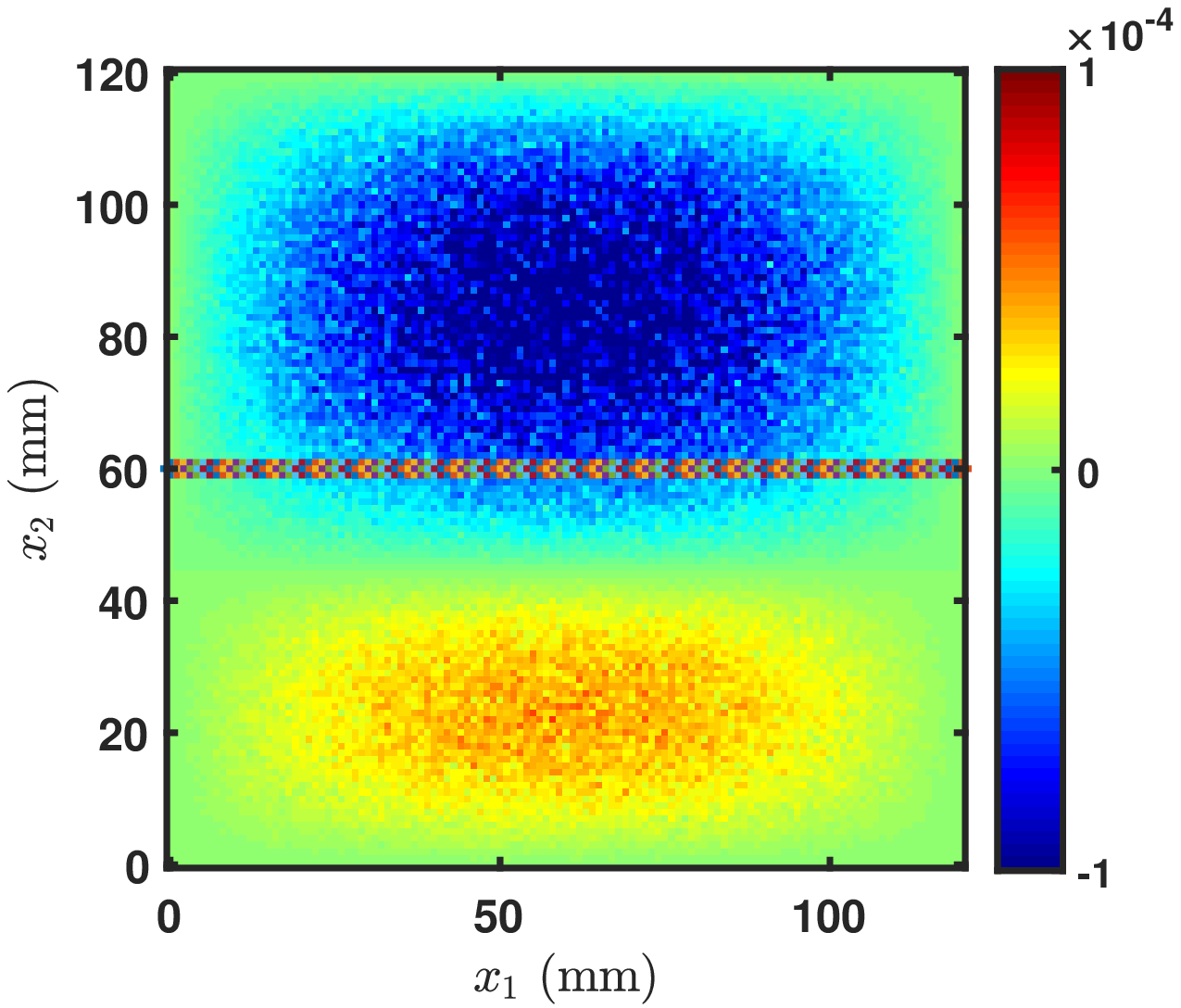}
	\caption{Noisy simulated data: (a) real part of  $u$ (mm); (b) imaginary part of  $u$ (mm).}\label{fig::lobs20}
\end{figure}
\begin{figure}[H]
	\centering
	\includegraphics[width=.6\textwidth]{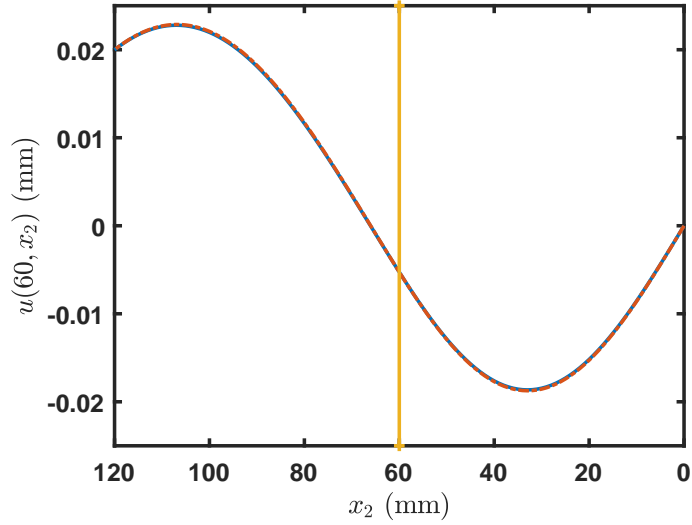}
	\caption{Simulated data along $x_1=60$ mm (blue line: without noise, red line: with 20$\%$ relative noise).}\label{fig::lu60}
\end{figure}

By applying our Levenberg-Marquad method, we recovered $\gamma$ from noisy data which is shown in Table \ref{table::lreconver} and the associated reconstructed wave fields are shown in Figure \ref{fig::lobsr} and Figure \ref{fig::lur60}. The result shows that both the storage modulus $G'$ and loss modulus $G''$ are recovered very well.
\begin{table}[H]
	\caption{Recovery of $\gamma$.}\label{table::lreconver}
	\centering
	\tabcolsep=8pt
	\begin{tabular}{c||c|c||c|c}
		\hline
		\hline	
		Initial guess & $G'$ &  30 kPa & $G''$ & 0.5 Pa$\cdot$s$\times \omega$\\
		\hline
		\hline
		Layer 1 ($x_L<x_2<120$) & $G'$ &  19.8202 kPa & $G''$ & 0.3849 Pa$\cdot$s$\times \omega$\\
		\hline
		Layer 2  ($0<x_2<x_L$)& $G'$ & 9.9829 kPa & $G''$ & 0.2990 Pa$\cdot$s$\times \omega$\\
		\hline
		\hline
	\end{tabular}
\end{table}
\begin{figure}[H]
	\centering
	(a)\includegraphics[width=.45\textwidth]{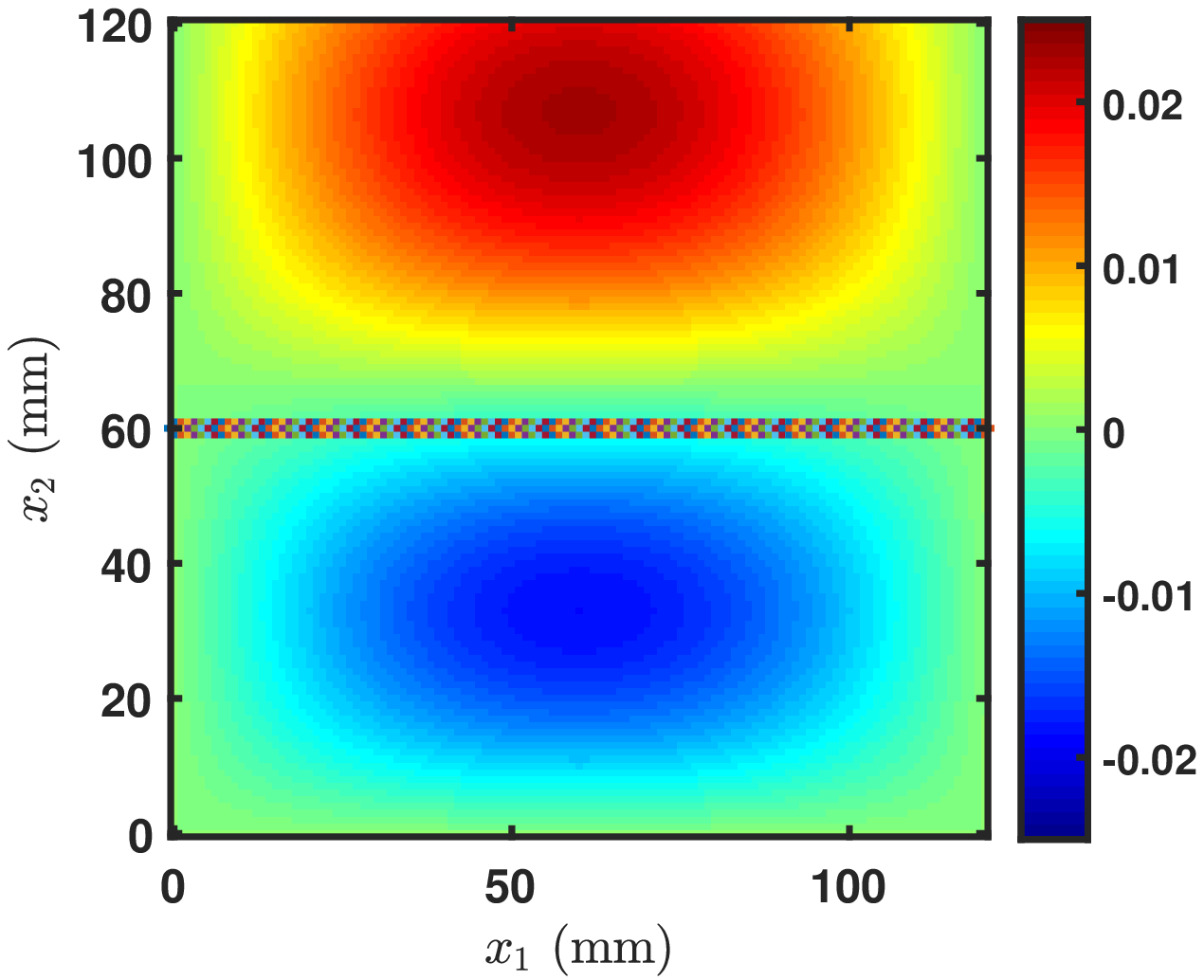}\,
	(b)\includegraphics[width=.45\textwidth]{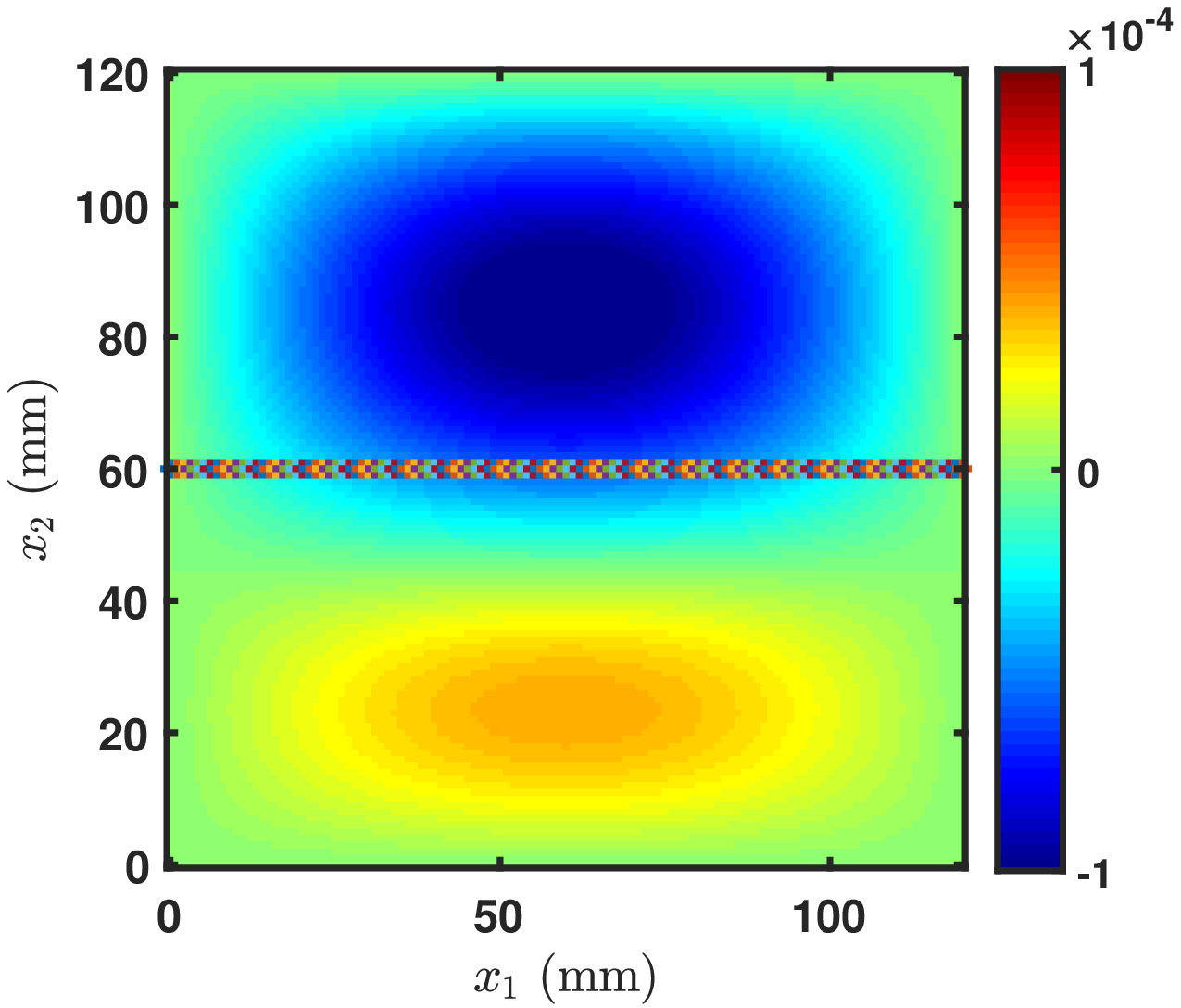}
	\caption{Reconstructed simulated data: (a) real part of  $u$ (mm); (b) imaginary part of  $u$ (mm).}\label{fig::lobsr}
\end{figure}
\begin{figure}[H]
	\centering
	\includegraphics[width=.6\textwidth]{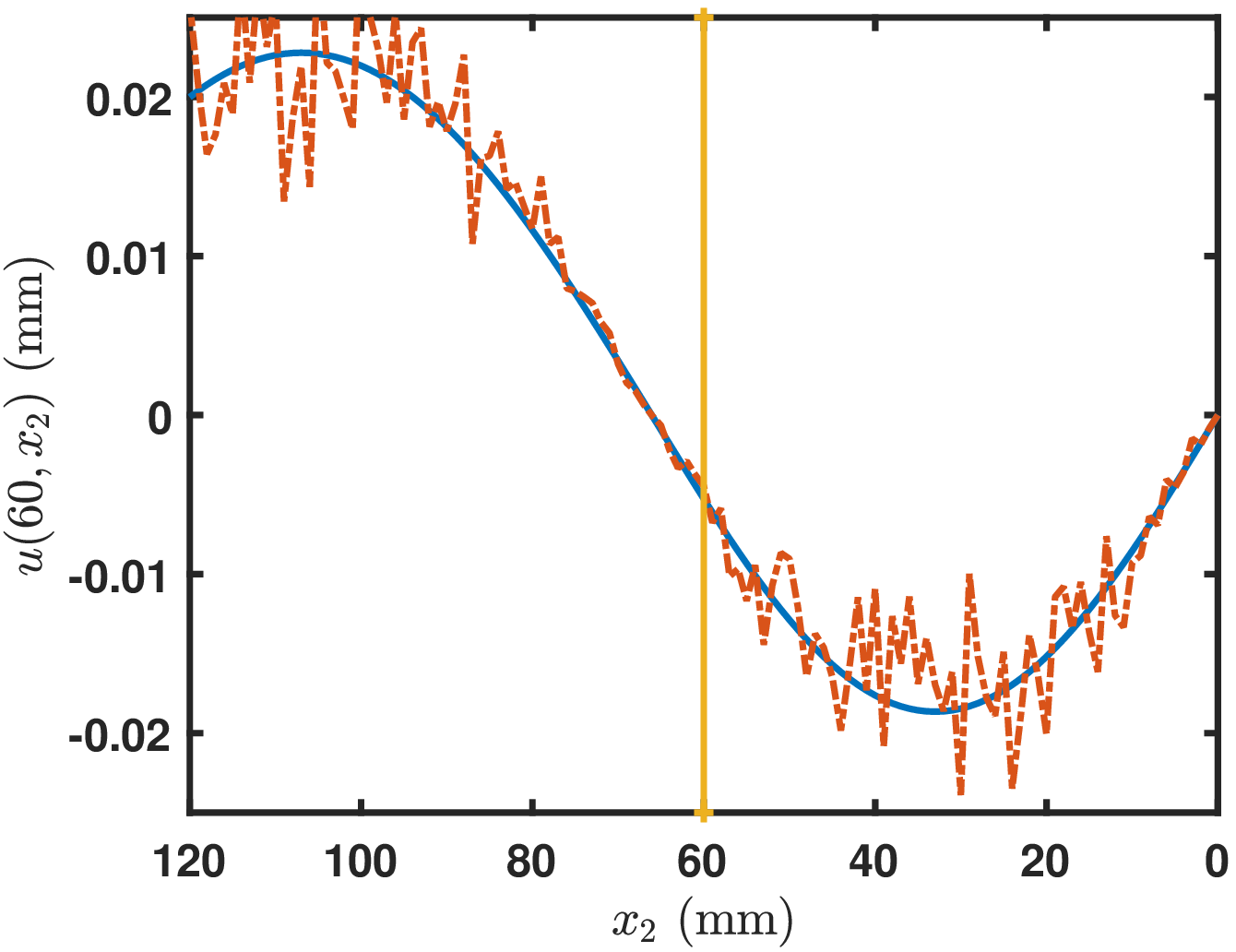}
	\caption{Reconstructed data along $x_1=60$ mm (blue line: without noise, red dot dash line: reconstructed).}\label{fig::lur60}
\end{figure}

\begin{figure}[H]
	\centering
	\includegraphics[width=.6\textwidth]{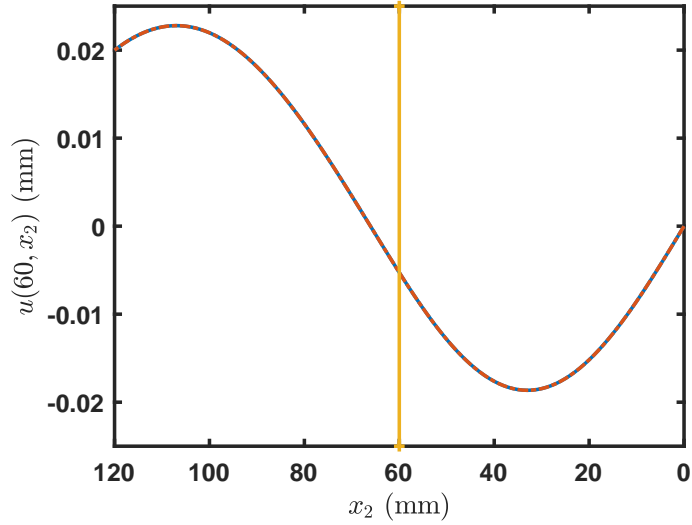}
	\caption{Simulated data along $x_1=60$ mm (real part of $u$, blue line: $G''=0$, red line: $G''\neq 0$).}\label{fig::uev60}
\end{figure}

It is interesting to observe here that the imaginary part of simulated data is quite small compared with the real part of simulated data which is almost the same as the one in Example 1.1 (see Figure \ref{fig::uev60}) . This is because the loss modulus $G''=0.4$ Pa$\cdot$s$\times\omega$$\approx$ 0.0503 kPa for layer 1 ($0.3$ Pa$\cdot$s$\times\omega$$\approx$ 0.0377 kPa for layer 2) is small while the storage modules is 20 kPa for layer 1 (10 kPa for layer 2). However, the existence of  loss modulus, no matter how small it is, enables us to choose some initial guess of $\gamma$ in Table \ref{table::lreconver} not so close to the exact value due to Theorem \ref{thm::convergence1} and Lemma \ref{lem::cone}. Further it is impossible to  recover neither $G'$ nor $G''$ reasonably well by using other existing methods, such as the modified integral method in \cite{JiangP} since there is less than half wave in each layer which cannot meet the requirement of our modified integral method. Similar result can either find in \cite{AmmN}.

\noindent {\small\bf Example 2.2: high frequency case:}
Let the angular frequency $\omega=250$ Hz. Use the simulated data without noise (see Figure \ref{fig::obs}) and the noisy simulated data with 20$\%$ relative Gaussian noise (see Figure \ref{fig::obs20} and Figure \ref{fig::u60}).
\begin{figure}[H]
	\centering
	(a)\includegraphics[width=.45\textwidth]{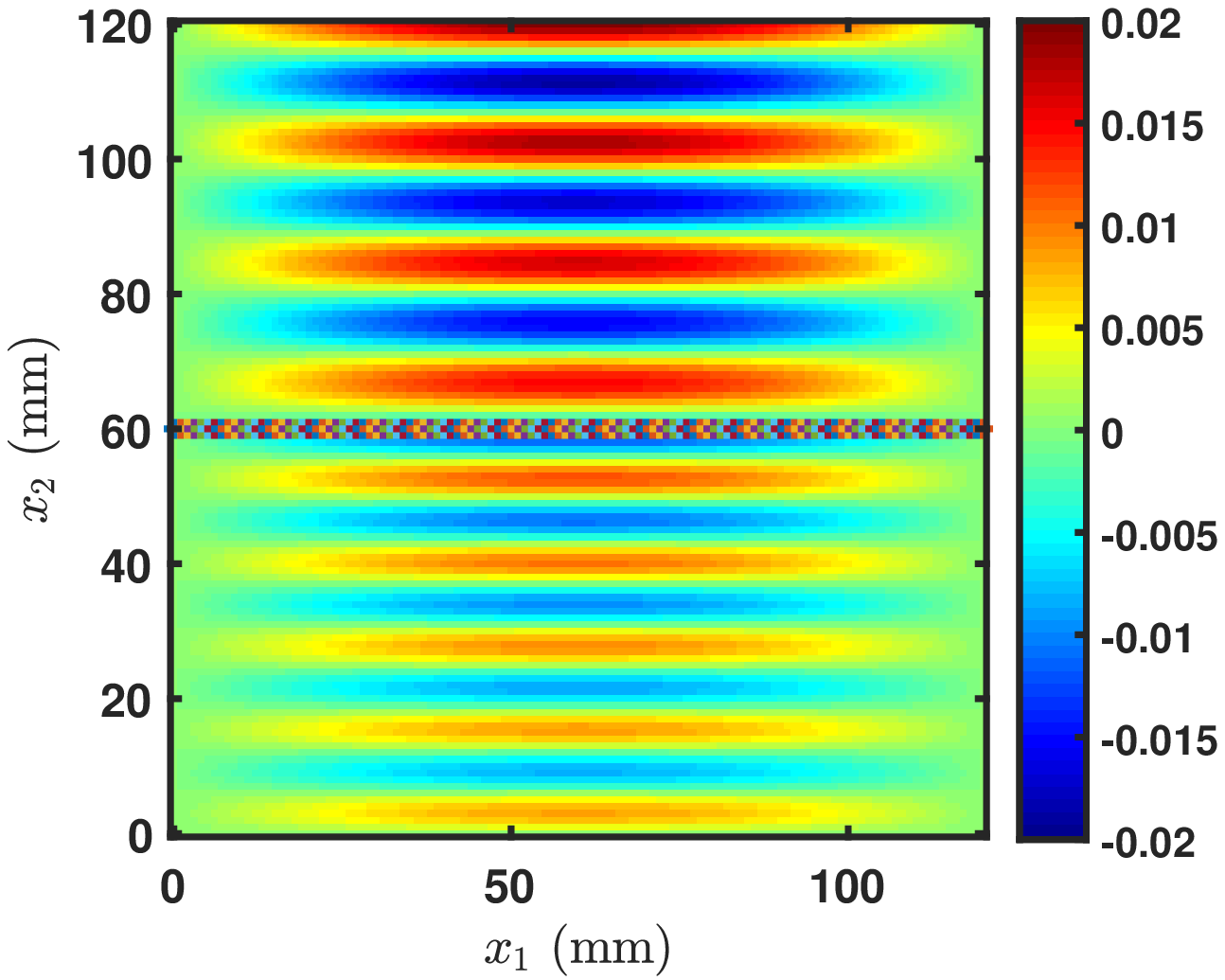}\,
	(b)\includegraphics[width=.45\textwidth]{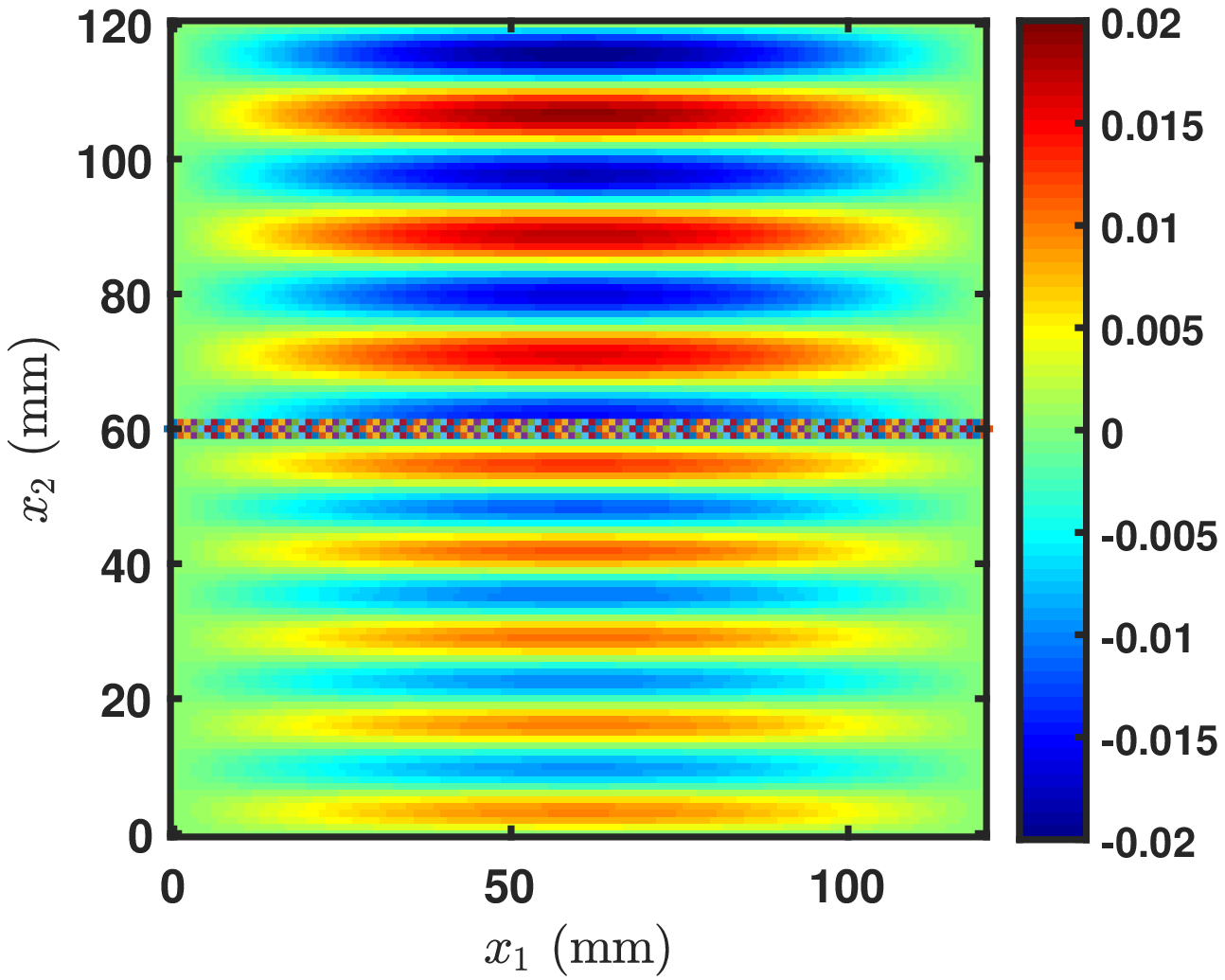}
	\caption{Simulated data without noise: (a) real part of  $u$ (mm); (b) imaginary part of  $u$ (mm).}\label{fig::obs}
\end{figure}
\begin{figure}[H]
	\centering
	(a)\includegraphics[width=.45\textwidth]{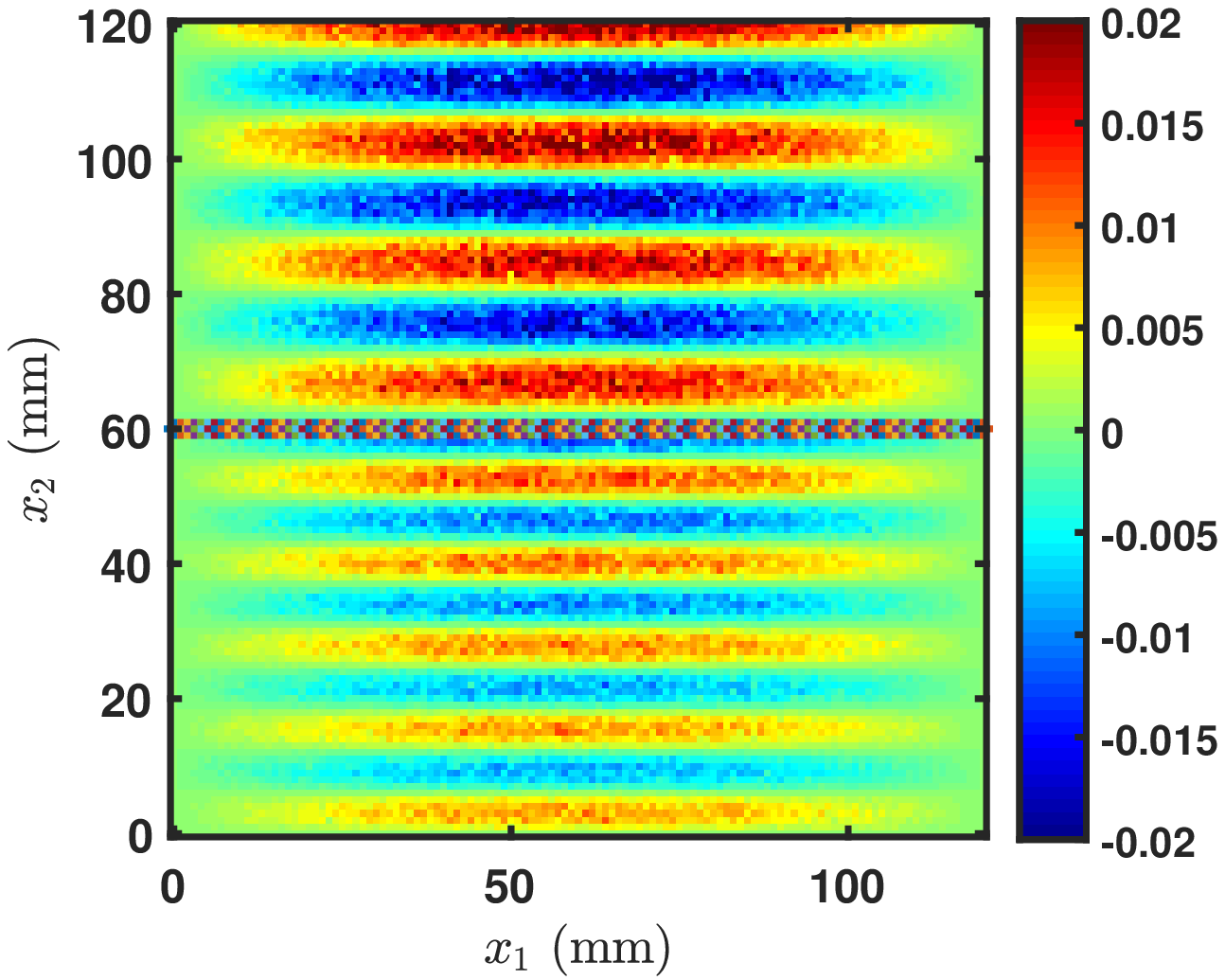}\,
	(b)\includegraphics[width=.45\textwidth]{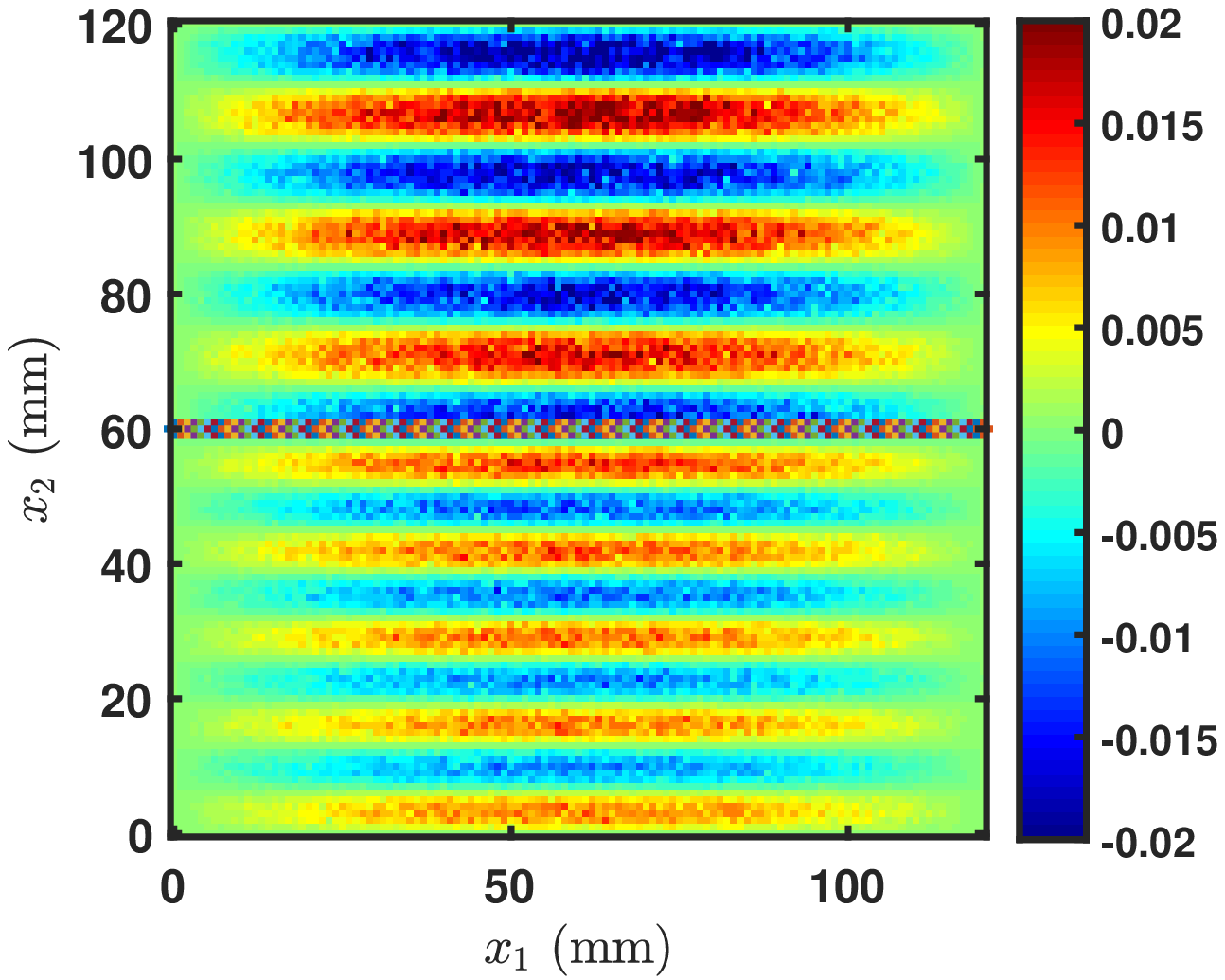}
	\caption{Noisy simulated data: (a) real part of  $u$ (mm); (b) imaginary part of  $u$ (mm).}\label{fig::obs20}
\end{figure}
\begin{figure}[H]
	\centering
	\includegraphics[width=.6\textwidth]{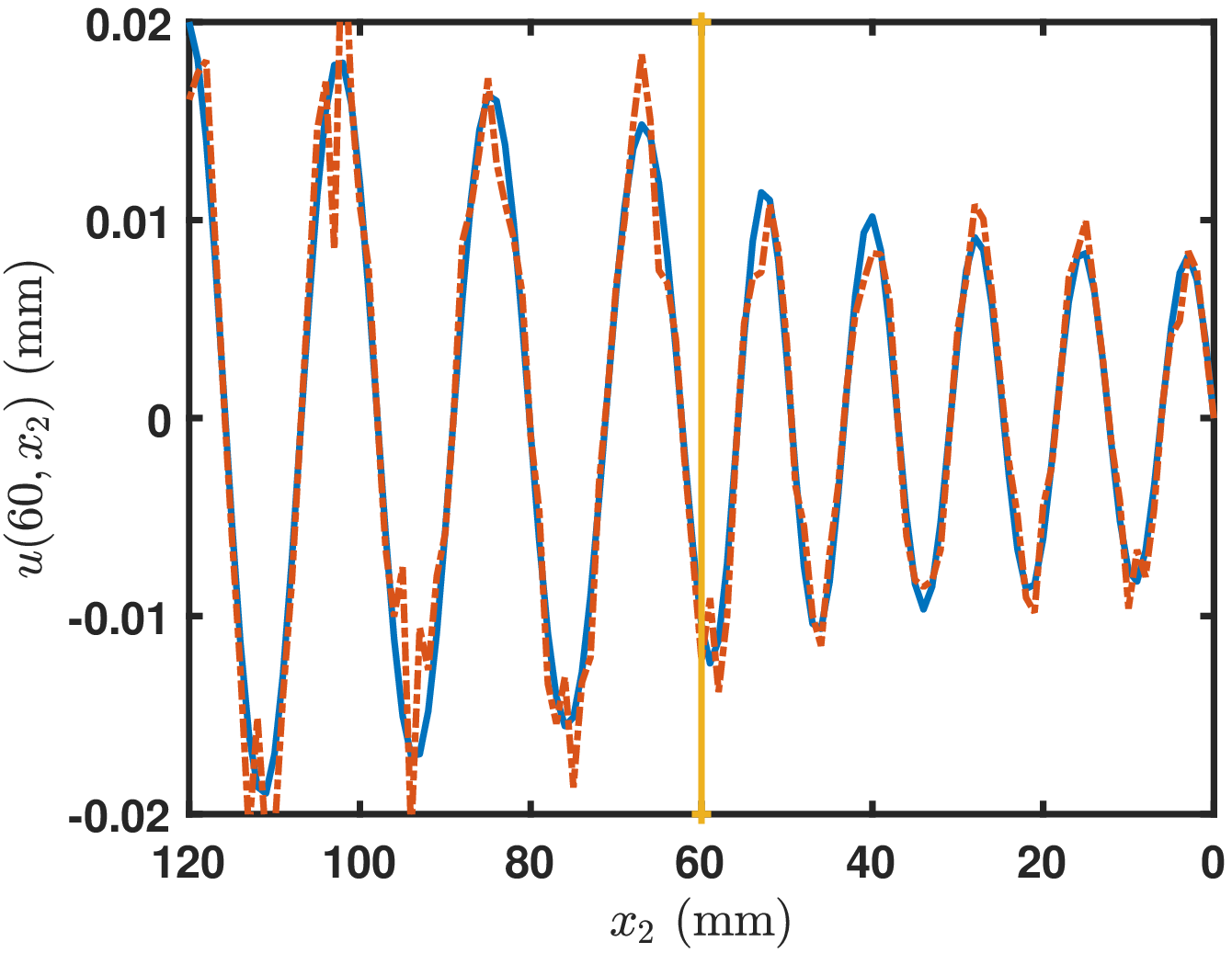}
	\caption{Simulated data along $x_1=60$ mm (blue line: without noise, red line: with 20$\%$ relative noise).}\label{fig::u60}
\end{figure}

The recovery of  $\gamma$ from noisy data is shown in Table \ref{table::reconver} and the reconstructed wave fields are shown in Figure \ref{fig::obsr} and Figure \ref{fig::ur60}. Again the recovery is very well. 
\begin{table}[H]
	\caption{Recovery of $\gamma$.}\label{table::reconver}
	\centering
	\tabcolsep=8pt
	\begin{tabular}{c||c|c||c|c}
		\hline
		\hline	
		Initial guess & $G'$ &  30 kPa & $G''$ & 0.5 Pa$\cdot$s$\times \omega$\\
		\hline
		\hline
		Layer 1 ($x_L<x_2<120$) & $G'$ &    19.9951 kPa & $G''$ & 0.3948 Pa$\cdot$s$\times \omega$\\
		\hline
		Layer 2  ($0<x_2<x_L$)& $G'$ & 9.9997 kPa & $G''$ & 0.3040 Pa$\cdot$s$\times \omega$\\
		\hline
		\hline
	\end{tabular}
\end{table}
\begin{figure}[H]
	\centering
	(a)\includegraphics[width=.45\textwidth]{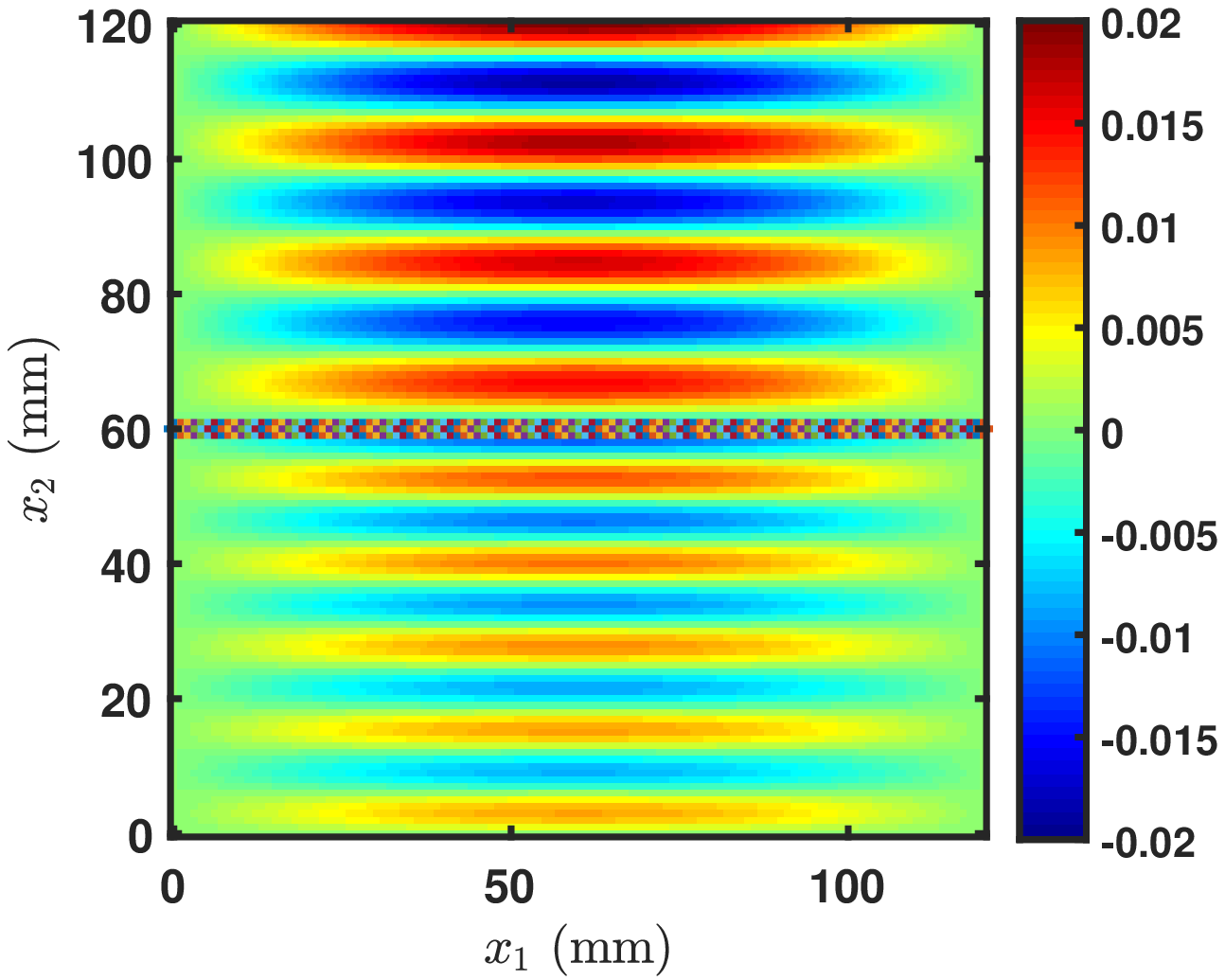}\,
	(b)\includegraphics[width=.45\textwidth]{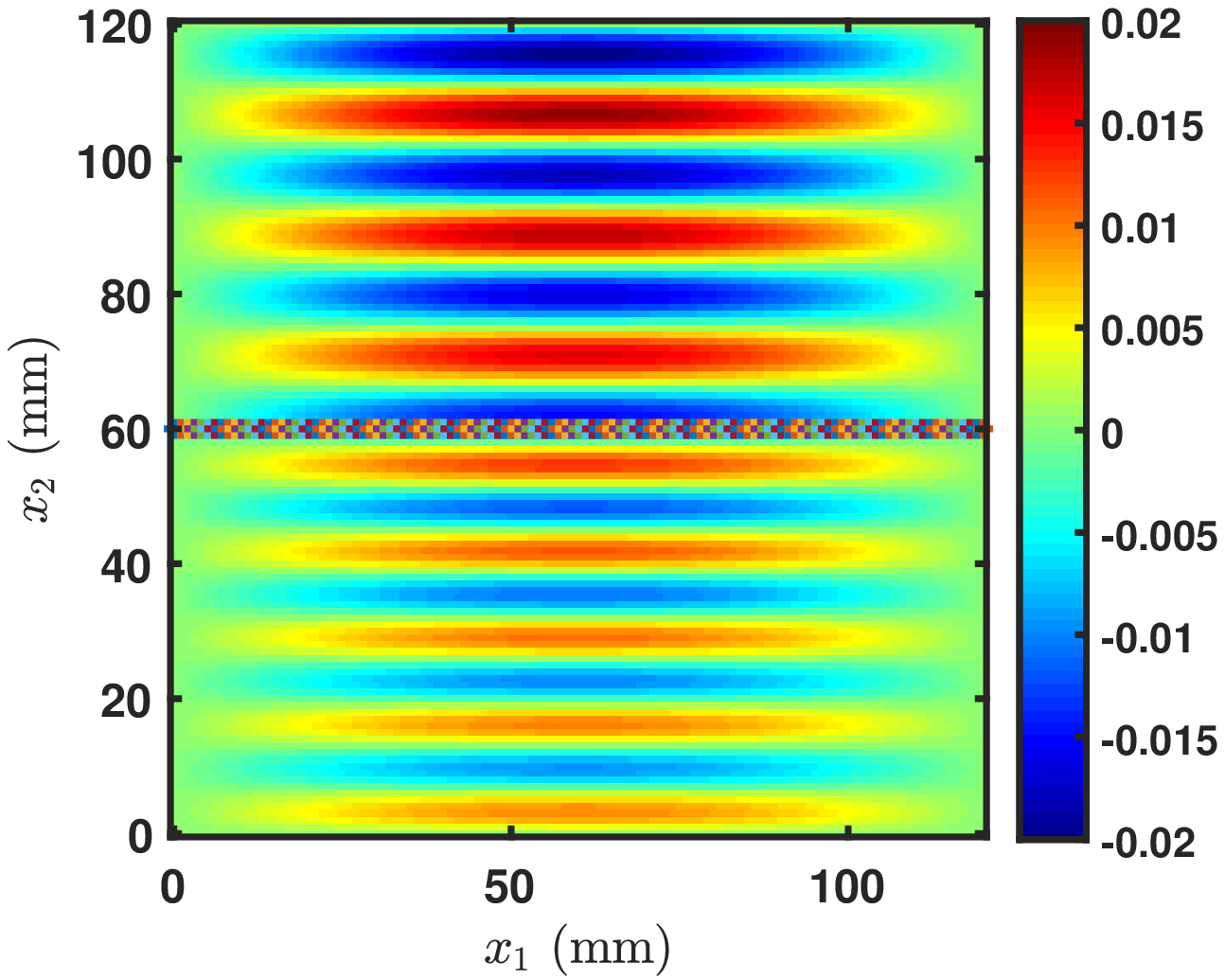}
	\caption{Reconstructed simulated data: (a) real part of  $u$ (mm); (b) imaginary part of  $u$ (mm).}\label{fig::obsr}
\end{figure}
\begin{figure}[H]
	\centering
	\includegraphics[width=.6\textwidth]{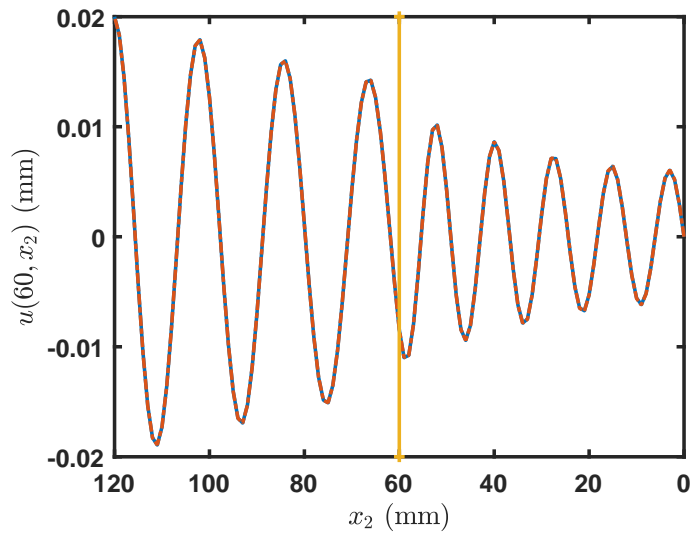}
	\caption{Reconstructed data along $x_1=60$ mm (blue line: without noise, red dot dash line: reconstructed).}\label{fig::ur60}
\end{figure}

Here, we applied the modified integral method to recover $\gamma$ from noisy data (see Figure \ref{fig::mim60}). The recovery of $G'$ is good, meanwhile the recovery of $G''$ is quite poor. 
\begin{figure}[H]
	\centering
	(a)\includegraphics[width=.45\textwidth]{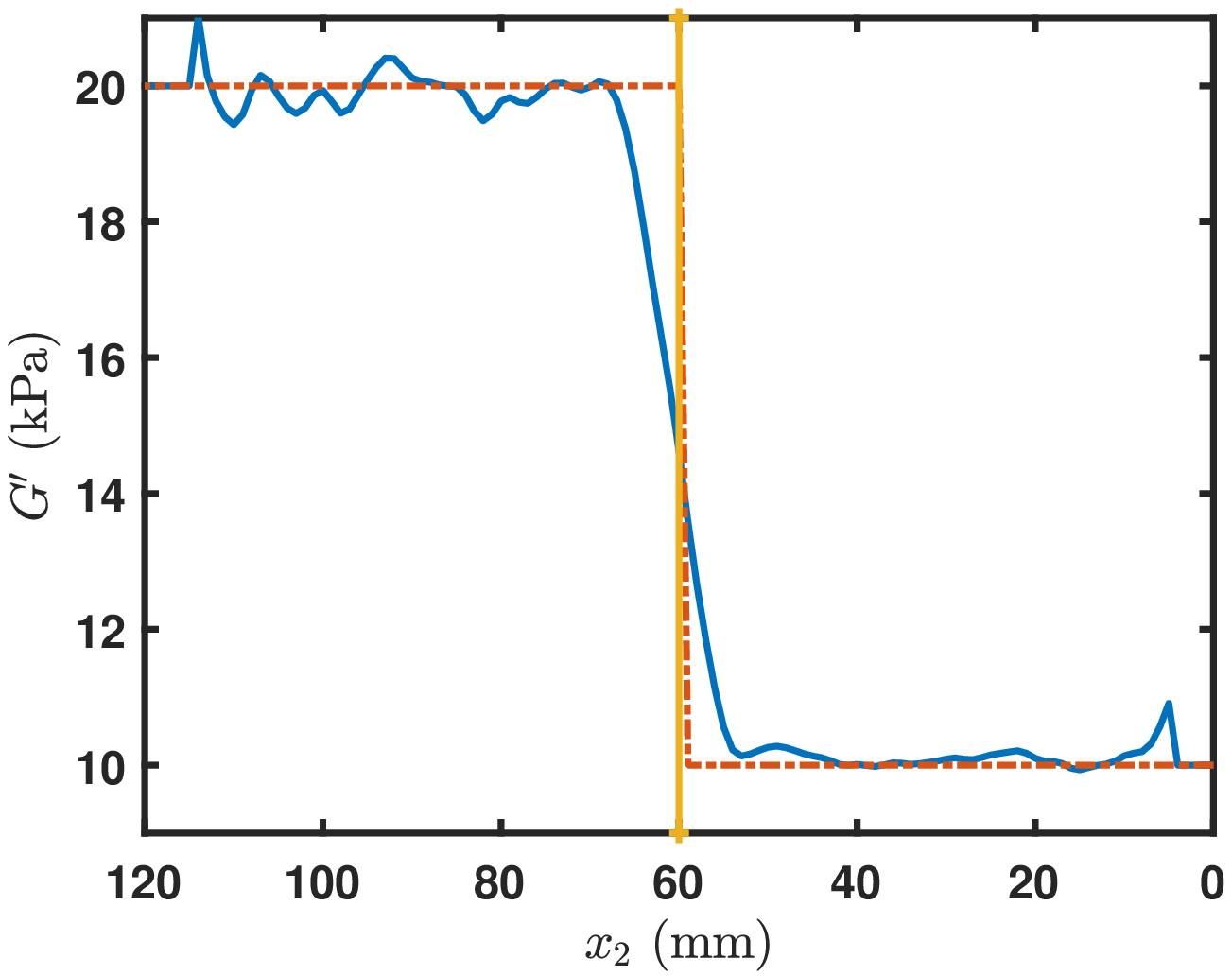}\,
	(b)\includegraphics[width=.45\textwidth]{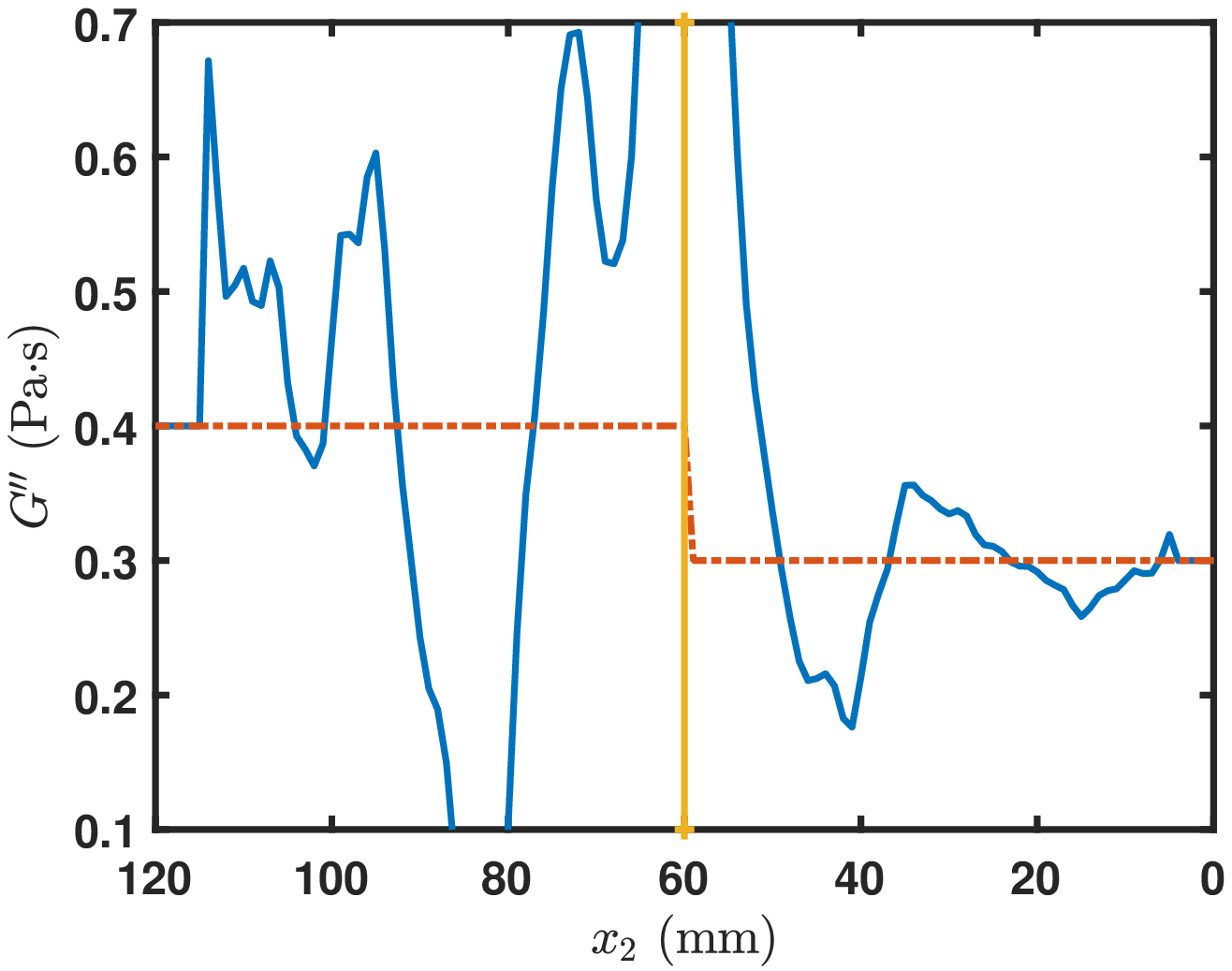}
	\caption{Recovery of $\gamma$ by using modified integral equation method along $x_1=60$ mm ((1) $G'$, (b) $G''$, blue line: recovery, red dot dash line: exact value).}\label{fig::mim60}
\end{figure}

\section{Discussions and conclusion}
We have shown that the convergence of Levenberg-Marquard method for an inverse problem with single interior measurement which arises in the data analysis for the magnetic resonance elastography  (MRE). The key for this was to show that the measurement map of MRE satisfies the tangential cone condition. A similar result can be obtained also for other coefficients identification problem by single interior measurement of solution to the boundary value problem for partial differential equation of divergence form such that the real part or the imaginary part of the associated sesquilinear form is positive. The numerical performance of this method was given for several cases and observed that it is quite well.
In particular the recovery of the loss modulus was very good compared with the other existing methods.
Based on these we conclude that this method has a strong potential to become one of a standard method for elastogram not only recovering the storage modulus but also the loss modulus.

\end{document}